\documentclass[leqno,10pt]{report}

\usepackage{amsmath}
\usepackage{amsthm}
\usepackage{amssymb}
\usepackage{amscd}
\usepackage{amsfonts}
\usepackage{amsbsy}
\usepackage{array}
\usepackage{enumerate}
\usepackage{amsthm}
\usepackage{xcolor}
\newtheorem{theorem}{Theorem}[section]
\newtheorem{proposition}[theorem]{Proposition}
\newtheorem{lemma}[theorem]{Lemma}
\newtheorem{corollary}[theorem]{Corollary}
\newtheorem{definition}[theorem]{Definition}
\newtheorem{remark}[theorem]{Remark}
\newtheorem{example}[theorem]{Example}

\begin{document}

\title{Symmetries of  Ordinary Differential Equations:\\
A Short Introduction}

\author{
Sebastian Walcher\\
Lehrstuhl A f\"ur Mathematik, RWTH Aachen\\
D-52056 Aachen, Germany\\
 \\
\copyright\quad S. Walcher 2019, 2023
}

\date{}

\maketitle

\begin{abstract}
These lecture notes provide an introduction to the theory and application of symmetry methods for ordinary differential equations, building on minimal prerequisites. Their primary purpose is to enable a quick and self-contained approach for non-specialists; they are not intended to replace any of the monographs on this topic. The content of the notes lies ``transversal'' to most standard texts, since we put emphasis on autonomous ODEs of first order, which often play a relatively minor role in the literature. Moreover, we avoid the technical build-up involving jet spaces and prolongations of group actions.\\
The notes are organized as follows. Chapter 1 contains basic material on ODEs with analytic right hand side, most of which may be known to the reader. Some remarks address differential equations with smooth right hand side, pointing out similarities and differences. Chapter 2 introduces to Sophus Lie's classical theory of local one-parameter (orbital) symmetry groups in the context of first order (mostly autonomous) equations, with a short digression to second order equations. Chapter 3 deals with ``multi-parameter symmetries'', including a clarification of notions, and with invariant sets that are (in some way) forced by symmetries. In Chapter 4 the table is turned: We start with a given (linear) group and discuss differential equations that admit this group as a group of symmetries, with an emphasis on toral groups. Finally we consider Poincar\'e-Dulac normal forms as a special class admitting toral symmetry groups, and focus on some of their special properties. While Chapters 1 - 3 contain no new material (apart perhaps from organization and presentation), there is some less known and some new material in Section 4. An Appendix (Section 5) contains, for quick reference, a summary of facts on power series and analytic functions. The theoretical results are accompanied and illustrated by some (rather small and elementary) examples.\\
A short list of basic references is given, but -- in line with the introductory nature of these notes -- there is no attempt to provide a complete list, which would be longer than the present text. We refer to the bibliographies in the cited papers and books. \\
These notes were originally written to accompany a series of lectures at Shanghai Jiaotong University in September and October 2019, during a pleasant and productive research visit. The author thanks the School of Mathematical Sciences for the opportunity to work at SJTU, and in particular Xiang Zhang and his colleagues and students at the Dynamical Systems Group for their hospitality.\\
The 2023 update contains some minor stylistic changes and correction of typos. Moreover the list of references has been updated.

\end{abstract}

\tableofcontents
\chapter{Basic notions and facts}
This chapter has preparatory character. Required are some knowledge of the theory of ordinary differential equations and, later on, of properties of local analytic functions and vector fields. (See Chapter 5 for background material on the latter.) We recall some basic definitions and facts and present a number of relevant constructions.\\

Throughout these notes $\mathbb K$ stands for $\mathbb R$ or $\mathbb C$, and $U\subseteq \mathbb K^n$ is a nonempty, open and connected set. The algebra of analytic $\mathbb K$-valued functions on $U$ will be denoted by $A(U)$, and the vector space of analytic $\mathbb K^n$-valued functions on $U$ will be denoted by $\mathcal{A}(U)$.
We consider an analytic autonomous ordinary differential equation
\begin{equation}\label{ode}
\dot x =\frac{dx}{dt}= f(x)
\end{equation}
defined on $U$.
Note that we  will always assume the independent variable $t$ to be real.
\section{Existence, uniqueness, dependence}
\begin{theorem}\label{eudthm}
\begin{enumerate}[(a)]
\item Every initial value problem $\dot x=f(x)$, $x(0)=y\in U$ has a unique solution $F(t,y)$ on an interval $I_{\rm max}(y)$, which cannot be extended beyond this interval.
\item The set
\[
\widetilde U:=\bigcup_{y\in U}\left(I_{\rm max}(y)\times \{y\}\right)\subseteq\mathbb R\times U
\]
is open and the map
\[
F:\,\widetilde U\to U,\quad (t,y)\mapsto F(t,y)
\]
is analytic.
\item Local one-parameter group property: One has
\[
\begin{array}{rcl}
F(0,y)&=&y\\
F(t_1+t_2,y)&=&F(t_1,F(t_2,y))
\end{array}
\]
for all $y\in U$ and all $t_1,\,t_2\in \mathbb R$ such that $F(t_2,y)$ and either $F(t_1,F(t_2,y))$ or $F(t_1+t_2,y)$ are defined. \\In particular $F(-t,F(t,y))=y$ for all $t\in I_{\rm max}(y)$.
\end{enumerate}
\end{theorem}
We call $F$ the {\em general solution} (or the {\em local flow}) of \eqref{ode}, and we will sometimes refer to $f$ as an {\em infinitesimal generator} of the {\em local one-parameter transformation group} $F$.

Occasionally we will also consider non-autonomous ordinary differential equations
\begin{equation}\label{nonaut}
\dot x=q(t,x) 
\end{equation}
with $q$ defined and analytic on a nonempty, open and connected subset of $\mathbb R\times\mathbb K^n$, and corresponding initial value problems. One may ``autonomize'' such an equation via
\begin{equation}\label{autnonaut}
\begin{array}{rcl}
\dot \xi&=&1\\
\dot x&=&q(\xi,x),
\end{array}
\end{equation}
which allows to carry over, cum grano salis, results about autonomous systems to the non-autonomous case.
\begin{example}{\em
The partial derivative $D_2F(t,y)$ of the local flow of \eqref{ode} with respect to the variable $y$ satisfies a non-autonomous linear differential equation, the so-called {\em variational equation}
\[
\frac{\partial }{\partial t}D_2F(t,y)=Df(F(t,y)D_2F(t,y),\quad D_2F(0,y)=I
\]
with $I$ denoting the identity matrix. This is readily verified by differentiating $\frac{\partial }{\partial t}F(t,y)=f(F(t,y))$.}
\end{example}
\section{Lie derivative}
\begin{definition}\label{liederdef}
Given $f$ as in \eqref{ode},  the corresponding Lie derivative
$X_f$ acts on $A(U)$ via
\[
\phi\mapsto X_f(\phi),\quad X_f(\phi)\,(x):=D\phi(x)\,f(x).
\]
\end{definition}
This definition is motivated by the observation
\begin{equation}\label{liedermot}
\frac d{dt}\phi(F(t,y))=X_f(\phi)(F(t,y)).
\end{equation}
We note some properties of Lie derivatives.
\begin{proposition}\label{liederprop} Let $f\in\mathcal{A}(U)$. Then the following hold:
\begin{enumerate}[(a)]
\item $X_f=0\Leftrightarrow f=0$.
\item $X_f$ is a derivation of $A(U)$, thus is linear and satisfies the ``product rule''
\[
X_f(\phi\cdot\psi)=X_f(\phi)\cdot\psi+\phi\cdot X_f(\psi) \text{  for all  } \phi,\,\psi\in A(U).
\]
\item {\em Lie series identity:} For all $\phi\in A(U)$ one has 
\[
\phi(F(t,y))=\exp (tX_f)(\phi)(y)=\sum_{k\geq 0}\frac{t^k}{k!}X_f^k(\phi)(y).
\]
\end{enumerate}
\end{proposition}
\begin{proof}
Parts (a) and (b) are straightforward. To verify part (c), note that analyticity implies the existence of an expansion
\[
\phi(F(t,y))=\sum_{k\geq 0}t^k\psi_k(y)
\]
with analytic functions $\psi_k$, and moreover 
\[
k!\,\psi_k(y)=\frac{\partial^k}{\partial t^k}\phi(F(t,y))|_{t=0}.
\]
Now use \eqref{liedermot} and induction.
\end{proof}
We remark in passing that every derivation of $A(U)$ has the form $X_f$ for some analytic $f$. In $\geq 20^{\rm th}$ century language, these derivations are called vector fields. We will frequently identify $f$ and $X_f$, and also call $f$ a vector field.
\section{Lie bracket}
The commutator of two derivations (of any algebra) is again a derivation; the derivations therefore form a Lie algebra. This observation motivates the Lie bracket of $f,\,g\in\mathcal{A}(U)$: For $\phi\in A(U)$ we have the identities
\[
X_g(\phi)(x)=D\phi(x)\,g(x)\Rightarrow X_fX_g(\phi)(x)=D^2\phi(x)(g(x),f(x))+D\phi(x)Dg(x)f(x).
\]
Reversing the order of $f$ and $g$ and using the symmetry of the second derivative $D^2\phi(x)$, one obtains the identity
\begin{equation}\label{liebident}
(X_fX_g-X_gX_f)(\phi)(x)=D\phi(x)\left(Dg(x)\,f(x)-Df(x)\,g(x)\right).
\end{equation}
\begin{definition}\label{liebdef}
The {\em Lie bracket} of $f,\,g\in\mathcal{A}(U)$ is defined by
\[
\left[f,\,g\right](x):=Dg(x)f(x)-Df(x)g(x).
\]
Moreover we define 
\[
{\rm ad}\,f:\,\mathcal{A}(U)\to \mathcal{A}(U); \quad g\mapsto\left[f,\,g\right].
\]
\end{definition}
We collect properties of the Lie bracket:
\begin{proposition}\label{liebproperties}
\begin{enumerate}[(a)]
\item For all $f,\,g\in\mathcal{A}(U)$ one has
\[
X_{\left[f,g\right]}=X_fX_g-X_gX_f.
\]
\item With the standard vector space operations and the bracket $\left[\cdot,\cdot\right]$, $\mathcal{A}(U)$ is a Lie algebra, thus the bracket is bilinear, antisymmetric and satisfies the {\em Jacobi identity}
\[
\left[ f,\,\left[g,\,h\right]\right]+\left[ h,\,\left[f,\,g\right]\right]+\left[ g,\,\left[h,\,f\right]\right]=0.
\]
\item For all $f,g\in\mathcal{A}(U)$ and all $\psi\in A(U)$ one has the identity
\[
\left[f,\,\psi\,g\right]=X_f(\psi)\,g +\psi\left[f,\,g\right].
\]
\end{enumerate}
\end{proposition}
\begin{proof} Part (a) is a restatement of \eqref{liebident}, while part (b) follows from (a), the Lie algebra structure of the derivations of $A(U)$ and Proposition \ref{liederprop}(a). Part (c) follows from the definitions by straightforward calculation.
\end{proof}
\section{Solution-preserving maps (morphisms)}
It is advisable to discuss mathematical structures together with structure-preserving maps. For autonomous ODEs this leads to
\begin{definition}\label{solpresdef} Let \eqref{ode} be given, and $\emptyset\not=U^*\subseteq U$ open. Moreover let $V\subseteq\mathbb K^m$ be nonempty, open and connected, and $g\in\mathcal{A}(V)$.\\
 A map $\Phi: U^*\to V$ is called {\em solution-preserving} from \eqref{ode} to $\dot x=g(x)$ if $\Phi$ maps every solution of \eqref{ode} in $U^*$ to a solution of $\dot x=g(x)$; thus
\[
\Phi(F(t,y))=G(t,\Phi(y)) \text{  for all  }y\in U^*,\,t\in I_{\rm max}(y) \text{  with  } F(t,y)\in U^*.
\]
\end{definition}
A priori we do not require $\Phi$ to be analytic (or smooth) here, since there are interesting continuous (but not differentiable) solution preserving maps. We note a simple criterion when $\Phi$ is at least differentiable.
\begin{lemma}\label{solprescrit}
Let $\Phi$ and $g$ be as in Definition \ref{solpresdef}, and $\Phi$ differentiable. Then $\Phi$ is solution-preserving from \eqref{ode} to $\dot x=g(x)$ if and only if the identity
\begin{equation}\label{solpresid}
D\Phi(x)\,f(x)=g(\Phi(x))
\end{equation}
holds on $U^*$.
\end{lemma}
\begin{proof}
For one direction, differentiate the identity $\Phi(F(t,y))=G(t,\Phi(y))$ with respect to $t$ and set $t=0$. Conversely, assume that the identity \eqref{solpresid} holds. Then 
\[
\frac{\partial}{\partial t} \Phi(F(t,y))=D\Phi(F(t,y))f(F(t,y))=g(\Phi(F(t,y)),
\]
whence $ \Phi(F(t,y))$ solves $\dot x=g(x)$.
\end{proof}
\begin{remark}{\em The notion of solution-preserving maps includes coordinate transformations: Let $\widehat U\subseteq \mathbb K^n$ be open, connected and nonempty, and $\Psi:\,\widehat U\to U$ analytic such that $D\Psi(x)$ is invertible for all $x\in \widehat U$. Define
\[
f^*(x):=D\Psi(x)^{-1}f(\Psi(x)) \text{  for  }x\in \widehat U.
\]
Then $\Psi$ is locally invertible and (by construction) solution-preserving from $\dot x=f^*(x)$ to $\dot x=f(x)$. One may interpret $\Psi$ as a coordinate transformation, and interpret $\dot x=f^*(x)$ as \eqref{ode} being expressed in new coordinates.}
\end{remark}
We next show compatibility of solution-preserving maps with Lie derivatives and Lie brackets.
\begin{proposition}\label{liecompatible}
\begin{enumerate}[(a)]
\item Let $\Phi$ and $g$ be as in Definition \ref{solpresdef}, and $\Phi$ analytic. Then for any $\rho\in A(V)$ one has
\[
X_f(\rho\circ\Phi)=X_g(\rho)\circ\Phi.
\]
\item For all $f_1,\,f_2\in \mathcal{A}(U)$ and $g_1,\,g_2\in \mathcal{A}(V)$ such that $D\Phi(x)\,f_i(x)=g_i(\Phi(x))$, $i=1,2$, one also has
\[
D\Phi(x)\left[f_1,\,f_2\right](x)=\left[g_1,\,g_2\right](\Phi(x)).
\]
\end{enumerate}
\end{proposition}
\begin{proof} Part (a) follows directly from the definitions and Lemma \ref{solprescrit}, and part (b) is most easily proven as a consequence of part (a) and Proposition \ref{liebproperties}(a).
\end{proof}
The following result (called the ``straightening theorem'') is not hard to prove but it is fundamental. It will turn out to be of practical value in the computation of symmetry reductions, and it contributes to the theory by showing that the local behavior of ordinary differential equations is very simple (and uninteresting) near any non-stationary point.
\begin{theorem}\label{thmstraight}
Let \eqref{ode} be given, and let $y_0\in U$ such that $f(y_0)\neq 0$. Then there exists a neighborhood $U^*$ of $y_0$ and a locally invertible analytic map $\Psi:U^*\to U$ that is solution-preserving from $\dot{x} =\begin{pmatrix}
1 \\ 0 \\ \vdots \\ 0
 \end{pmatrix}$ to $\dot{x} = f(x)$.
\end{theorem}
\begin{proof}
By employing an affine coordinate transformation we may assume that $y_0=0$, $f(y_0) = \begin{pmatrix}
1 \\ 0 \\ \vdots \\ 0
 \end{pmatrix}$.   
Define $\Psi(x):= F\Bigl(x_1,\begin{pmatrix}
0 \\ x_2 \\ \vdots \\ x_n
 \end{pmatrix}\Bigr)$. By Theorem \ref{eudthm} this is analytic near $0$. With $\frac{\partial}{\partial t}F(t,y) = f(F(t,y))$ one has furthermore
\[
 D\Psi(x)\cdot \begin{pmatrix}
1 \\ 0 \\ \vdots \\ 0
 \end{pmatrix} = \frac{\partial \Psi}{\partial x_1}(x) = f\Bigl(F \Bigl(x_1,
\begin{pmatrix}
0 \\ x_2 \\ \vdots \\ x_n
 \end{pmatrix}\Bigr)\Bigr) = f(\Psi(x));
\]
thus $\Psi$ is solution-preserving as asserted.\\
Moreover $\frac{\partial \Psi}{\partial x_1}(0) = f(0) = \begin{pmatrix}
1 \\ 0 \\ \vdots \\ 0
 \end{pmatrix}$, and from $\Psi\bigl( \begin{pmatrix}
0 \\ x_2 \\ \vdots \\ x_n
 \end{pmatrix}\bigr) = \begin{pmatrix}
0 \\ x_2 \\ \vdots \\ x_n
 \end{pmatrix}$ one finds that
\[
D\Psi(0) = \begin{pmatrix}
            1 & \ast & \cdots & \ast   \\
	    0 & 1 & & _{\displaystyle{0}}\quad           \\
	    \vdots & & \ddots &         \\
	    0 & \quad ^{\displaystyle{0}} & & 1
           \end{pmatrix}\,.
\]
Invertibility of the linear map $D\Psi(0)$ implies local invertibility of $\Psi$.
\end{proof}
\begin{remark}{\em The proof of the straightening theorem is, up to a point, constructive: If an explicit expression for the local flow $F(t,y)$ is known then the straightening map can be determined explicitly.}

\end{remark}

To finish this section, we establish a Lie series identity involving Lie brackets.
\begin{proposition}\label{liebrackser} Let $f, \,h\in \mathcal{A}(U)$. 
\begin{enumerate}[(a)]
\item  Then
\[
 \frac{\partial}{\partial s}\left(D_2H(s,y)^{-1} f(H(s,y))\right) = D_2 H(s,y)^{-1}\left[h,f\right](H(s,y)).
\]
for all $y$ and all $s$ in the maximal existence interval for $\dot x=h(x),\,x(0)=y$.
\item One has the identity
\[
D_2H(s,y)^{-1} f(H(s,y))=\exp(s\,{\rm ad}\,h)(f)(y)=\sum_k \frac{s^k}{k!}\,({\rm ad}\,h)^k(f)(y).
\]
\end{enumerate}
\end{proposition}
\begin{proof} For part (a), first recall the variational equation
\[
 \frac{\partial}{\partial s}(D_2H(s,y)) = Dh(H(s,y))\cdot D_2H(s,y).
\]
Moreover the inversion map $\iota:X\to X^{-1}$ of linear automorphisms of $\mathbb K^n$ is rational, hence analytic and satisfies the identity $D_{\iota}(X)Y = -X^{-1} YX$.   \\
Using these facts, as well as the product rule and the chain rule, one obtains
\begin{align*}
 & \frac{\partial}{\partial s}\left(D_2H(s,y)^{-1} f(H(s,y))\right)  \\
& = -D_2H(s,y)^{-1} \left(\frac{\partial}{\partial s}D_2H(s,y)\right)D_2H(s,y)^{-1}   f(H(s,y)) \\
&\quad+ D_2H(s,y)^{-1} Df(H(s,y))\cdot \left(\frac{\partial}{\partial s}H(s,y)\right)  \\
& = -D_2 H(s,y)^{-1} Dh (H(s,y)) f(H(s,y)) + D_2H(s,y)^{-1} Df(H(s,y)) \cdot h(H(s,y)).
\end{align*}
For part (b), by analyticity there exists an expansion
\[
D_2H(s,y)^{-1} f(H(s,y))=\sum_k s^k g_k(y)
\]
with analytic $g_k$, and
\[
k\,! g_k(y)=\frac{\partial^k}{\partial s^k}\left(D_2H(s,y)^{-1} f(H(s,y))\right)|_{s=0}.
\]
Now use part (a) and induction.
\end{proof}
\section{Invariant sets}
A subset $Y$ of $U$ is called \emph{invariant} for $\dot{x} = f(x)$ if $y\in Y$ implies $F(t,y)\in Y$ for all $t\in I_{\rm max}(y)$. In other words, a subset of $U$ is invariant if and only if it is a union of solution trajectories of \eqref{ode}. Thus the union, intersection and set-theoretic difference of invariant sets are invariant, and from dependence properties one sees that the closure, interior and boundary of a invariant set are themselves invariant.

The following criterion is frequently useful.
\begin{proposition}\label{invarcritprop}  
\begin{enumerate}[(a)]
\item Let $\varphi_1,\ldots,\varphi_r\in A(U)$. If there exist analytic functions $\mu_{ij}$ $(1\leq i,j\leq r)$ on $U$ such that
\[
 X_f(\varphi_i) = \sum_j\mu_{ij}\varphi_j\quad \text{for} \;\; 1\leq i\leq r\,, 
\]
then the common zero set $Y$ of $\varphi_1,\ldots,\varphi_r$ is invariant for $\dot{x}= f(x)$.
\item For a partial converse, assume $\mathbb K=\mathbb C$, and let $Y\subseteq U$ be a nontrivial analytic invariant set for \eqref{ode}. For fixed $y_0\in Y$ denote by $J(Y)$ the vanishing ideal of $Y$ in the algebra of (germs of) analytic functions at $y_0$. Let $\psi_1,\ldots,\psi_s$ be generators of $J(Y)$. Then there exist (germs of) analytic functions $\nu_{ij}$ such that
\[
X_f(\psi_i)=\sum_j \nu_{ij}\psi_j\quad \text{for} \;\; 1\leq i\leq r.
\]
\end{enumerate}
\end{proposition}
\begin{proof}
For part (a) let $y\in Y$ and abbreviate $z(t):=F(t,y)$. By assumption we have
\[
 \frac{d}{dt}\left(\varphi_i(z(t))\right)=\sum_j\mu_{ij}(z(t))\cdot \varphi_j(z(t)) \quad (1\leq i\leq r);
\]
therefore $w(t):= \begin{pmatrix}
                \varphi_1(z(t))  \\
		\vdots  \\
		\varphi_r(z(t))
                 \end{pmatrix}$ solves a homogeneous linear differential equation, and $y\in Y$ implies $w(0)=0$. By uniqueness one has $w(t)=0$ for all $t$.

To prove part (b), let $y$ be in a suitable neighborhood of $y_0$ and $\psi_1(y)=\cdots=\psi_s(y)=0$. Invariance implies $\psi_i(F(t,y))=0$ for all $t$ near $0$, and by differentiation and \eqref{liedermot} one finds $X_f(\psi_i)(y)=0$. The Hilbert-R\"uckert Nullstellensatz (see the Appendix) now shows $X_f(\psi_i)\in J(Y)$, and the assertion follows.
\end{proof}
\begin{remark}{\em Analogous statements, with analogous proofs, hold for invariant algebraic varieties of polynomial differential equations.}
\end{remark}
We introduce (resp. recall) special names.
\begin{definition} Let $\psi\in A(U)$ be nonconstant. Then:
\begin{enumerate}[(i)]
\item $\psi$ is called a {\em semi-invariant} of \eqref{ode} if there exists $\mu\in A(U)$ such that $X_f(\psi)=\mu\,\psi$.
\item $\psi$ is called a {\em first integral} of \eqref{ode} if $X_f(\psi)=0$.
\end{enumerate}

\end{definition}
By Proposition \ref{invarcritprop} the vanishing set of a semi-invariant, as well as every level set of a first integral, is invariant for $\dot x=f(x)$. Note that we require semi-invariants and first integrals to be nonconstant. As for the existence of first integrals, one has:
\begin{lemma}\label{locfirstint}
Let \eqref{ode} be given, and $y_0\in U$ with $f(y_0)\not=0$. Then there exists an open neighborhood $U^*$ of $y_0$ and functionally independent analytic first integrals $\psi_j:\,U^*\to \mathbb K$, $1\leq j\leq n-1$ of $\dot x=f(x)$.
\end{lemma}
\begin{proof}
For the case of constant $f = \begin{pmatrix}
1 \\ 0 \\ \vdots \\ 0
 \end{pmatrix}$, the assertion is obvious with $\psi_j=x_{j+1}$, $1\leq j\leq n-1$. The general statement follows with the straightening theorem and Proposition \ref{liecompatible}.
\end{proof}
On the other hand, (nonconstant) first integrals will not in general exist near stationary points.
\section{Solutions and solution orbits (trajectories)}

We distinguish between solutions  $F(t,y)$ (including parameterization) and solution orbits 
\[
\left\{F(t,y);\,t\in I_{\rm max}(y)\right\}
\]
(also called trajectories) of an autonomous differential equation \eqref{ode}. This leads to various notions of equivalence. In this section we make frequent use of properties of (local) analytic functions; see the Appendix for background information.
\begin{definition}\label{deforbequiv} Let $f,\,f^*\in\mathcal{A}(U)$.
\begin{enumerate}[(i)]
\item We call $f$ and $f^*$ (or rather, the corresponding differential equations) {\em orbit equivalent} if they have the same solution orbits on $U$.
\item We call $f$ and $f^*$ {\em locally orbit equivalent} near $y_0$ if they have the same solution orbits in a neighborhood $U^*$ of $y_0$.
\item We call $f$ and $f^*$ {\em generically orbit equivalent} if they have the same local solution orbits in an open-dense subset of $U$.
\end{enumerate}
\end{definition}
One verifies easily that the above define equivalence relations. There is one distinguished equivalence class which contains just $f=0$ (recall that $U$ is connected). In the following we will consider nonzero vector fields.
\begin{proposition}\label{proporbequiv} Let $f,\,f^*\in\mathcal{A}(U)$.
\begin{enumerate}[(a)]
\item Given $y_0\in U$ with $f(y_0)\not=0$, $f$ and $f^*$ are locally orbit equivalent near $y_0$ if and only if there is an open neighborhood $U^*$ of $y_0$ and an analytic $\mu:\,U^*\mapsto \mathbb K$ without zeros such that $f^*=\mu\cdot f$. 
\item $f$ and $f^*$ are locally orbit equivalent near some $y_0\in U$ with $f(y_0)\not=0$ if and only if $f$ and $f^*$ are generically orbit equivalent.
\item Let $\mathbb K=\mathbb C$ and assume that the sets of stationary points of $f$ resp. $f^*$ contain no submanifold of codimension one. Then $f$ and $f^*$ are generically orbit equivalent if and only if they are orbit equivalent.
\end{enumerate}
\end{proposition}
\begin{proof}We let
\[
 f=\begin{pmatrix}
  f_1 \\ \vdots \\ f_n
  \end{pmatrix}\quad \text{and} \quad f^*=\begin{pmatrix}
  f^*_1 \\ \vdots \\ f^*_n
  \end{pmatrix}\,. 
\]
(a) Let $f$ and $f^*$ be orbit equivalent near $y_0$. By the straightening theorem we may assume that 
$
 f=\begin{pmatrix}
  1 \\0\\  \vdots \\ 0
  \end{pmatrix}$, hence every solution trajectory of $\dot x=f(x)$ lies in a one dimensional affine subspace given by $x_2={\rm const.},\ldots,\, x_n={\rm const.}$ The same holds for the trajectories of $\dot x=f^*(x)$, which implies
$
 f^*=\begin{pmatrix}
  f^*_1 \\0\\  \vdots \\ 0
  \end{pmatrix}$.  For the reverse direction, let $z(t)$ be a solution of $\dot{x} = f(x)$; $f(z(0))\neq 0$, and $f^* = \mu f$. For the solution of $\dot{x} = f^*(x)$, $x(0) = z(0)$ make the ansatz $z(\tau(t))$, with $\tau(t)$ to be determined. Substitute to obtain
\[
 \dot{\tau}(t) f(z(\tau(t))) = \dot{z}(\tau(t))\cdot \dot{\tau}(t) = \frac{d}{dt}(z(\tau(t))) = f^*(z(\tau(t))) = \mu(z(\tau(t)))\cdot f(z(\tau(t)))\,.
\]
This holds if and only if $\dot{\tau}(t) = \mu(z(\tau(t)))$, $\tau(0) = 0$, which is a differential equation for $\tau$ in $\mathbb R$, with analytic right hand side. This shows the existence of $\tau$, and moreover $\dot{\tau}(0) = \mu(z(0))\neq 0$, hence $\tau$ has an inverse function near $0$.

As for part (b), assume that $f^*=\mu \, f$, with $\mu$ having no zeros on some neighborhood $U^*$ of $y_0$. If $f(y_0)=0$ then $f^*(y_0)=0$, and both (stationary) trajectories are equal. Otherwise, assume $f_1(y_0)\not=0$ with no loss of generality. Then $f_1^*=\mu\,f_1$ and therefore 
\[
f_1^*\,f=f_1\,f^*
\]
on $U^*$. By connectedness and the identity theorem for analytic functions, this holds on all of $U$. This shows the nontrivial part of the asserted equivalence.

To prove the nontrivial assertion of part (c), we may again assume that $f_1\not=0$ and $f_1^*\,f=f_1\,f^*$ on $U$. Let $z_0$ such that $f_1(z_0)=0$. Since the local ring of analytic functions at $z_0$ is a unique factorization domain, cancelling common factors of $f_1$ and $f_1^*$ yields $\sigma^*f=\sigma f^*$ with relatively prime $\sigma,\,\sigma^*$. Assume now that $\sigma(z_0)=0$, and let $Y$ be an irreducible component of the zero set of $\sigma$. Since $\sigma^*$ does not vanish identically on $Y$, $f$ must vanish identically on $Y$; a contradiction. We find $\sigma(z_0)\not=0$ and, by the same argument, $\sigma^*(z_0)\not=0$. Part (a) now shows local orbit equivalence at every point, which implies orbit equivalence on $U$.
\end{proof}
We note a consequence of the proof of part (b):
\begin{corollary}
 Let $f,\,f^*\in\mathcal{A}(U)$, both nonzero. Then $f$ and $f^*$ are generically orbit equivalent if and only if in a  neighborhood of some nonstationary point of $f$, every analytic first integral of $f$ is a first integral of $f^*$, and vice versa.
\end{corollary}
Finally, a word of caution:
\begin{remark}\label{onedumbrem}
{\em In dimension one, all nonzero analytic vector fields are generically orbitally equivalent.}
\end{remark}
\section{Further remarks and notes}
\begin{itemize}
\item Some basic references on the matters presented should be listed here: For Theorem \ref{eudthm} see e.g. Cartan \cite{Ca}; for Lie series one may consult Groebner \cite{Grob}. Lie derivatives and Lie brackets are essential tools in differential geometry; the introductory pages of Helgason \cite{Hel} provide more information (in the smooth setting). Occasionally we need some facts from Algebraic Geometry; Shafarevich \cite{Sha}, Ruiz \cite{Rui} and Zariski/Samuel \cite{ZaSa60} are good sources for these. (See also the facts collected in the Appendix.)
\item Most papers and monographs on symmetries nowadays -- other than in Lie's times --  deal with smooth (i.e., infinitely differentiable) functions and vector fields, rather than analytic ones. For the present notes, we chose to discuss the analytic case, which in some aspects provides a more satisfactory theory. We give a quick account of corresponding (and non-corresponding) results for the smooth case.\\
Theorem \ref{eudthm} also holds, mutatis mutandis, for smooth systems, and there is no problem in extending the definitions of Lie derivative and Lie bracket to the smooth case. As for Propositions \ref{liederprop} and \ref{liebproperties}, parts (a) and (b) hold as well, but parts (c) do not make sense in the smooth case. All the results from Section 1.4 up to and including Proposition \ref{liebrackser} remain correct in the smooth case, as does part (a) of Proposition \ref{invarcritprop} concerning invariant sets. The remainder of Section 1.5 again works for smooth functions and vector fields. In Sction 1.6 one may transfer the notions to the smooth case, but few general results remain, since many arguments rely on special properties of analytic functions.
\end{itemize}
\chapter{Local one-parameter (orbital) symmetry groups}

Sophus Lie seems to have been the first to observe the relevance of symmetries (in particular of local one-parameter groups of symmetries) for the investigation of differential equations, including reduction of dimension and finding explicit solutions. We will distinguish symmetries from orbital symmetries; Lie's work focussed on the latter.
\section{Symmetries}
\begin{definition} Let the system \eqref{ode} be given.
\begin{itemize}
\item A {\em symmetry} of \eqref{ode} is a solution-preserving map $\Phi$ (defined on some open $U^*\subseteq U$) from $\dot x=f(x)$ to itself, such that the derivative $D\Phi(x)$ is invertible for all $x\in U^*$. 
\item Moreover $h\in\mathcal{A}(U)$ is called an {\em infinitesimal symmetry} of \eqref{ode} if $h$ generates a local one-parameter group of symmetries of \eqref{ode}, thus every $H(s,\cdot)$, with $s$ near $0$, is a symmetry.
\end{itemize}
\end{definition}
We first characterize the infinitesimal generators of local one-parameter symmetry groups.
\begin{proposition}
Let $h\in \mathcal{A}(U)$. Then $h$ generates a local one-parameter group of symmetries for $\dot x=f(x)$ if and only if $\left[h,\,f\right]=0$.
\end{proposition}
\begin{proof}
By Lemma \ref{solprescrit}, $H(s,\,\cdot)$ is a symmetry of \eqref{ode} if and only if 
\[
D_2H(s,y)^{-1}f(H(s,y))= f(y)
\]
for all $y\in U$ and $s$ near $0$. Proposition \ref{liebrackser} shows that the latter is equivalent to $\left[h,\,f\right]=0$.
\end{proof}
We turn to the reduction of symmetric systems.
\begin{theorem}\label{symredthm}
Let $f,\,h\in\mathcal{A}(U)$ with $\left[h,\,f\right]=0$. Moreover let $\emptyset \neq W\subseteq \mathbb K^n$ be open and let the analytic map $\Psi:W\to U$ be solution-preserving from $\dot{x} = \begin{pmatrix}
1 \\ 0 \\ \vdots \\ 0     \end{pmatrix}$ to $\dot{x} = h(x)$, with $D\Psi(x)$ invertible for all $x$. Define
\[
f^*(x):= D\Psi(x)^{-1}f(\Psi(x)).
\]
Then $\Psi$ is solution-preserving from $\dot{x} =f^*(x)$ to $\dot x = f(x)$, and 
\[
f^*(x) = \begin{pmatrix}
                          f_1^*(x_2,\ldots, x_n)  \\
			  \vdots    \\
			  f_n^*(x_2,\ldots, x_n)
                         \end{pmatrix}
\]
does not depend on $x_1$.
\end{theorem}
\begin{proof}
$\Psi$ is solution-preserving by construction. Since solution-preserving maps respect Lie brackets (see Proposition \ref{liecompatible}), one obtains
$\left[\begin{pmatrix} 1 \\ 0 \\ \vdots \\ 0\end{pmatrix},f^*\right]=0$, or $\frac{\partial f^*}{\partial x_1}= 0$.
\end{proof}
Therefore a solution of the system
\begin{align*}
 \dot x_1 & = f_1^*(x_2,\ldots,x_n)   \\
 \dot x_2 & = f_2^*(x_2,\ldots,x_n)   \\
	& \qquad \vdots    \\
 \dot x_n & = f_n^*(x_2,\ldots,x_n)
\end{align*}
can be obtained by solving a differential equation in $\mathbb K^{n-1}$ and an additional quadrature. (In the present context one decrees that quadratures are unproblematic.) Thus, in effect one has reduced the dimension by one. In particular,  for dimension $n=2$, with a separable equation for one variable, only quadratures remain. \\
There is another approach to reduction. 
\begin{proposition}\label{redbyinv}
Let $f,\,h\in\mathcal{A}(U)$ with $\left[h,\,f\right]=0$. Then:
\begin{enumerate}[(a)]
\item For every $\gamma\in A(U)$ with $X_h(\gamma)=0$ one also has $X_h\left(X_f(\gamma)\right)=0$.
\item Let $U^*\subseteq U$ be nonempty and connected, and assume that $\gamma_1,\ldots,\gamma_s:\,U^*\to\mathbb K$ satisfy the following properties:
\begin{itemize}
\item Each $\gamma_i$ is a first integral of $h$.
\item For every $i$, $1\leq i\leq n$ and every $y_0\in U^*$ there exists an analytic function $\sigma_i$ in $s$ variables such that $X_f(\gamma_i)(x)=\sigma_i (\gamma_1(x),\ldots,\gamma_s(x))$ in some neighborhood of $y_0$.
\end{itemize}
Then the map
\[
\Gamma:=\begin{pmatrix}\gamma_1\\ \vdots\\ \gamma_s\end{pmatrix}:\, U^*\to \mathbb K^s
\]
is solution-preserving from $\dot x=f(x)$ to some differential equation on an open subset of $\mathbb K^s$.
\end{enumerate}
\end{proposition}
\begin{proof}
Part (a) follows with
\[
0=X_{\left[h,\,f\right]}(\gamma)=X_hX_f(\gamma)-X_fX_h(\gamma)=X_hX_f(\gamma)
\]
using Proposition \ref{liebproperties}. Part (b) is then automatic, since
\[
X_f(\gamma_i)=\sigma_i(\gamma_1,\ldots,\gamma_s), \quad 1\leq i\leq s
\]
may be restated as
\[
D\Gamma(x)f(x)=\begin{pmatrix}\sigma_1\\ \vdots\\ \sigma_s\end{pmatrix}\left(\Gamma(x)\right).
\]
\end{proof}
\begin{remark}{\em 
The hypothesis of part (b) above is certainly satisfied when for {\em every} first integral $\theta$ of $h$ on $U^*$ and every $y_0\in U^*$ there exists an analytic function $\sigma$ in $s$ variables such that $\theta(x)=\sigma(\gamma_1(x),\ldots,\gamma_s(x))$ near $y_0$. This is the case, for instance,  with $s=n-1$ in a suitable neighborhood of any $z$ with $h(z)\not=0$; see Lemma \ref{locfirstint}.}
\end{remark}
To apply Theorem \ref{symredthm} in a ``practical'' context, one needs to know a straightening transformation for $h$. According to the proof of Theorem \ref{thmstraight} this transformation may be determined explicitly from the general solution $H(t,\,y)$, and therefore systems with ``nice'' infinitesimal symmetries are amenable to the reduction method from Theorem \ref{symredthm}.
\begin{example}\label{rotexdim2}{\em 
Let $\mu_1$ and $\mu_2$ be arbitrary functions of one variable, and  $g(x):=\begin{pmatrix}
	   -x_2 \\ +x_1                                                             \end{pmatrix}$.   For
\[
f(x):=\mu_1(x_1^2+x_2^2)\cdot x+\mu_2(x_1^2+x_2^2)\cdot g(x)
\]
one verifies $[g,f]=0$. (Thus $f$ admits rotational symmetry.) Now with 
 \[
\Psi(x):=
\begin{pmatrix}
\arctan \frac{x_2}{x_1}	\\ x_1^2+x_2^2
\end{pmatrix}
, \quad D\Psi(x) = \begin{pmatrix}
				\frac{-x_2}{x_1^2+x_2^2} & \ \frac{x_1}{x_1^2+x_2^2}\\
				2x_1 & 2x_2
				\end{pmatrix}
\]
 one has
$D\Psi(x)g(x) = \begin{pmatrix}
		1 \\ 0
		\end{pmatrix}$. 
Transforming $f$, we get
\[
 D\Psi(x) f(x) = (x_1^2+x_2^2)\mu_1 (x_1^2+x_2^2)\cdot\begin{pmatrix}
                                                          0 \\ 2
                                                         \end{pmatrix}+ \mu_2(x_1^2+x_2^2)\cdot
\begin{pmatrix}
  1 \\ 0
\end{pmatrix}\,,
\]
thus $D\Psi(x)f(x) = \widetilde{f}(\Psi(x))$, with $\widetilde{f}(x) = 
\begin{pmatrix}
  \mu_2(x_2) \\ 2x_2\cdot \mu_1(x_2)
\end{pmatrix}$.   \\
Note that $\Psi$ is the inverse of a straightening function according to Theorem \ref{thmstraight}; therefore the strategy from Theorem \ref{symredthm} was modified. As expected, $\dot{x} = \widetilde{f}(x)$ reduces to an equation in $\mathbb K$ and one quadrature.\\
One can also proceed via Proposition \ref{redbyinv}, with the first integral $\gamma=x_1^2+x_2^2$, and 
\[
X_f(\gamma)=2\mu_1(\gamma)\cdot\gamma.
\]
}
\end{example}
Lie went on to perform a kind of ``reverse engineering'', by constructing differential equations with prescribed symmetries. The procedure will be described next.
\begin{remark}\label{revengrem}{\em 
Let $h\in \mathcal{A}(U)$, moreover  $ W\subseteq \mathbb K^n$ nonempty and open, and $\Psi:W\to U^*\subseteq U$ an analytic map that is solution-preserving from $\dot{x} = \begin{pmatrix}
1 \\ 0 \\ \vdots \\ 0     \end{pmatrix}$ to $\dot{x} = h(x)$, with $D\Psi(x)$ invertible for all $x$. For any $f^*\in\mathcal{A}(W)$ which depends only on $x_2,\ldots, x_n$, define $f\in \mathcal{A}(U)$ by $D\Psi(x)f^*(x)=f(\Psi(x))$. Then $\left[h,\,f\right]=0$, and $f$ admits a reduction via Theorem \ref{symredthm}. (To find $f$ explicitly, one needs to know explicitly the local inverse of $\Psi$.)\\
At first sight this procedure may seem strange, but with its help Lie managed to provide a unified approach to large classes of differential equations (primarily in dimension two) that were known to be solvable ``by some trickery'', and moreover to extend these classes.}
\end{remark}
\begin{example}\label{diagonalex2d}{\em 
Let $\beta\in\mathbb K$ and  
\[
h(x) = \begin{pmatrix}
                              x_1  \\ \beta\; x_2
                               \end{pmatrix},\quad h\left(\begin{pmatrix}
							1  \\  0
							\end{pmatrix}\right) = 
\begin{pmatrix}
1  \\  0
\end{pmatrix}.
\]
From the proof of the straightening theorem, with $H(t,y)=
\begin{pmatrix}
 e^{t} y_1  \\
e^{q/p\cdot t}y_2
\end{pmatrix}$, one finds that 
\[
\Psi (x) = H\left(x_1-1,\begin{pmatrix}
						1  \\ x_2
						\end{pmatrix}\right)=\begin{pmatrix}e^{x_1-1}   \\
	      e^{\beta(x_1-1)}x_2 \end{pmatrix}
\]
satisfies $D\Psi(x)\begin{pmatrix} 1\\0\end{pmatrix}=h(\Psi(x)$.
The inverse of $\Psi$ is easily computed as
\[
\Gamma(x)=\begin{pmatrix}1+\log x_1\\ x_1^{-\beta}x_2\end{pmatrix},\quad\text{ with  }D\Gamma(x)^{-1}=\begin{pmatrix}x_1 & 0\\ \beta x_1^2x_2 & x_1^\beta\end{pmatrix}.
\]
Now given $f^*(x)=\begin{pmatrix}f_1^*(x_2)\\ f_2^*(x_2)\end{pmatrix}$, one computes
\[
f(x)=D\Gamma(x)^{-1}f^*(\Gamma(x))=\begin{pmatrix}x_1f_1^*(x_1^{-\beta}x_2)\\ \beta x_2f_1^*(x_1^{-\beta}x_2)+x_1^\beta f_2^*(x_1^{-\beta}x_2)\end{pmatrix},
\]
and one may directly verify that $\left[f,h\right]=0$.
We thus have constructed a class of differential equations which admit a given infinitesimal symmetry (and locally we determined all of these equations).
}
\end{example}
As can be seen from the discussion above, it is useful to know sets of commuting vector fields (but trivial cases like $h\in\mathbb K f$ do not help). One construction method (see e.g.\ \cite{HaWa}, although there should exist an earlier reference for this result) is as follows.
\begin{lemma}\label{commlem}
Let $W\subseteq \mathbb K^n$ be nonempty and open, and $\Gamma:\,W\to \mathbb K^n$ analytic with invertible derivative $D\Gamma(x)$ for all $x\in W$. For any $a\in\mathbb K^n$ set
\[
q_a(x):=D\Gamma(x)^{-1}a.
\]
Then 
\[
\left[q_a,\,q_b\right]=0 \text{  for all  }a,\,b\in\mathbb K^n.
\]
\end{lemma}
\begin{proof}
Using the rule for differentiating matrix inversion, and with $D^2$ denoting the second derivative, one has
\[
D\left(D\Gamma(x)^{-1}a\right)z=-D\Gamma(x)^{-1}D^2\Gamma(x)(D\Gamma(x)^{-1}a,z).
\]
Therefore
\[
D\left(D\Gamma(x)^{-1}a\right)D\Gamma(x)^{-1}b=-D\Gamma(x)^{-1}D^2\Gamma(x)(D\Gamma(x)^{-1}a,D\Gamma(x)^{-1}b).
\]
Due to symmetry of the second derivative the last expression is symmetric in $a$ and $b$.
\end{proof}
Perhaps surprisingly, the reverse also holds:
\begin{proposition}\label{commprop}
Let $g_1,\ldots,g_n\in\mathcal{A}(U)$ such that $\left[g_i,\,g_j\right]=0$ for $1\leq i,\,j\leq n$, and let $y_0\in U$ such that $g_1(y_0),\ldots, g_n(y_0)$ are linearly independent in $\mathbb K^n$. Then there exist $W$ and $\Gamma$ as in Lemma \ref{commlem} such that $g_i(x)=D\Gamma(x)^{-1}e_i$, $1\leq i\leq n$, where the $e_i$ denote the elements of the standard basis of $\mathbb K^n$.
\end{proposition}
\begin{proof}
Define
\[
Q(x):=\left(g_1(x),\ldots,g_n(x)\right)^{-1}
\]
for all $x$ in a suitable neighborhood of $y_0$, and $q_a(x):=Q(x)^{-1}a$ for $a\in \mathbb K^n$ (hence $g_i=q_{e_i}$).
Using the identity
\[
D\left(Q(x)^{-1}\right)y=-Q(x)^{-1}\left(DQ(x)y\right)Q(x)^{-1}
\]
(again invoking the derivative of matrix inversion) we find that $\left[q_a,\,q_b\right]=0$ implies
\[
0=-Q(x)^{-1}\left(DQ(x)(Q(x)^{-1}b)Q(x)^{-1}a-DQ(x)(Q(x)^{-1}a)Q(x)^{-1}b\right).
\]
By invertibility of $Q(x)$ we obtain the closure condition
\[
\left(DQ(x)u\right)w=\left(DQ(x)w\right)u \text{  for all  } u,\,w\in \mathbb K^n;
\]
therefore $Q(x)$ is locally the derivative of some map $\Gamma$.
\end{proof}
\begin{example}\label{matrixriccex}{\em
In $GL(n,\mathbb K)\subseteq\mathbb K^{(n,n)}$ consider matrix inversion (which we here call by a different name) $\Gamma(x)=x^{-1}$. With the familiar rule one gets
\[
D\Gamma(x)y=-x^{-1}yx^{-1},\quad q_a(x)=D\Gamma(x)^{-1}a=xax.
\]
This observation explains why it is easy to solve $\dot x=xax$. Note that the polynomials $q_a$ appear as quadratic parts in the matrix Riccati equation. \\
Instead of matrix algebras one may work with an arbitrary associative unital algebra here.}
\end{example}
\section{Orbital symmetries}
Here we first need to clarify the notions.
\begin{definition}
Let $f\in\mathcal{A}(U)$, moreover let $V\subseteq \mathbb K^m$ open, connected and nonempty, and $g\in\mathcal{A}(V)$. 
\begin{enumerate}[(i)]
\item Given an open and nonempty $\widehat U\subseteq U$, and an analytic map $\Phi:\,\widehat U\to V$, we call $\Phi$ {\em orbit-preserving} from \eqref{ode} to $\dot x=g(x)$ if there exists $f^*\in\mathcal{A}(\widehat U)$ which is orbit equivalent to $f$ on $\widehat U$, such that $\Phi$ is solution-preserving from $\dot x=f^*(x)$ to $\dot x=g(x)$.
\item An {\em orbital symmetry} of \eqref{ode} is an orbit-preserving map $\Phi$ (defined on some open $\widehat U\subseteq U$) from $\dot x=f(x)$ to itself, with invertible derivative $D\Phi(x)$ for all $x\in\widehat U$. 
\item Moreover one says that $h\in\mathcal{A}(U)$ is an {\em infinitesimal orbital symmetry} of \eqref{ode} if $h$ generates a local one-parameter group of orbital symmetries of \eqref{ode}, thus every $H(s,\cdot)$, with $s$ near $0$, is an orbital symmetry.
\end{enumerate}
\end{definition}
\begin{remark}\label{orbpresrem}{\em
From Proposition \ref{proporbequiv} one sees:
\begin{itemize}
\item If $f\not=0$ and $g\circ\Phi\not=0$, then $\Phi$ is orbit-preserving from $\dot x=f(x)$ to $\dot x=g(x)$ if and only if there exist nonzero analytic $\rho,\,\sigma$ on $U$ such that
\[
\sigma(x)\,D\Phi(x)\,f(x)=\rho(x)\,g(\Phi(x)) \text{  for all  } x\in U.
\]
\item If $\mathbb K=\mathbb C$, and the zero set of $f$ contains no local component of codimension one, then $\Phi:\,U\to U$ is an orbital symmetry of \eqref{ode} if and only if there is a $\mu\in A(U)$ without zeros such that 
\[
\mu(x)\,D\Phi(x)\,f(x)=f(\Phi(x) \text{  for all  } x.
\]
\end{itemize}
}
\end{remark}
\begin{remark}{\em 
Orbit-preserving maps to a one-dimensional equation (in contrast to solution-preserving maps to dimension one) are not very interesting. Indeed, given $f\in\mathcal{A}(U)$, any map $\Phi:\, U\to \mathbb K$ with $X_f(\Phi)\not =0$ and any $g$ (defined on a suitable subset of $\mathbb K$) such that $g\circ \Phi\not=0$, one sees that the restriction of $\Phi$ to a suitable open subset of $U$ will be orbit-preserving to $\dot x=g(x)$. (Compare Remark \ref{onedumbrem}.)}
\end{remark}
\begin{proposition}\label{orbsymcharprop}
Let $h\in \mathcal{A}(U)$ and $\emptyset\not=\widehat U\subseteq U$ open. Then $h$ generates a local one-parameter group of orbital symmetries for $\dot x=f(x)$ on $\widehat U$  if and only if there is $\lambda\in A(\widehat U)$ such that 
\[
\left[h,\,f\right]=\lambda f.
\]
When $\mathbb K=\mathbb C$ and the zero set of $f$ has no local component of codimension one then $\lambda$ extends to an analytic function on $U$, and $h$ generates a local one-parameter group of orbital symmetries on all of $U$.
\end{proposition}
\begin{proof}
By the definitions and Remark \ref{orbpresrem}, $H(s,\,\cdot)$ is an orbital symmetry of \eqref{ode} if and only if 
\[
D_2H(s,y)^{-1}f(H(s,y))= \mu(s,y)f(y)
\]
for all $y\in U$ and $s$ near $0$, with $\mu$ necessarily analytic. For one direction of the proof, differentiate and use Proposition \ref{liebrackser}(a) to show that $\left[h,\,f\right]=\lambda f$, with $\lambda(y):=\frac{\partial}{\partial s}\mu(s,y)|_{s=0}$. For the reverse direction, use Proposition \ref{liebrackser}(b), Proposition \ref{liebproperties}(c) and induction, with
\[
\left[h,\,f\right]=\lambda f,\quad\left[h,\,\left[h,\,f\right]\right]=\left[h,\,\lambda f\right]=\left(X_h(\lambda)+\lambda^2\right) \,f,\quad\text{  etc.}
\]
\end{proof}
\begin{remark}{\em The case when there exists a one-codimensional set of stationary points must be excluded from the second statement of Proposition \ref{orbsymcharprop}, as the example $f(x)=x,\,h(x)=1$ and $\lambda(x)=1/x$ in dimension one shows.}
\end{remark}
It is natural to ask to what extent infinitesimal symmetries and infinitesimal orbital symmetries differ. The following result shows that (locally near nonstationary points, resp. generically) infinitesimal orbital symmetries are symmetries of an orbitally equivalent system.
\begin{proposition}\label{symvsorbsymprop}Let $f,\,h\in\mathcal{A}(U)$, and $\lambda\in A(U)$ such that $\left[h,\,f\right]=\lambda\, f$. Then:
\begin{enumerate}[(a)]
\item For every $y_0\in U$ with $h(y_0)\not=0$ there exist an open neighborhood $U^*$ and some $\sigma\in A(U^*)$ without zeros such that $\left[h,\,\sigma f\right]=0$.
\item If $\psi$ is a first integral of $h$, but not of $f$, on some open $\emptyset\not=\widehat U\subseteq U$ then
\[
\left[h,\frac1{X_f(\psi)}\,f\right]=0
\]
wherever $X_f(\psi)$ does not vanish.
\end{enumerate}
\end{proposition}
\begin{proof}
By Proposition \ref{liebproperties}(c), $\left[h,\,\sigma f\right]=0$ is equivalent to
\[
X_h(\sigma)+\lambda\sigma=0.
\]
This linear first order partial differential equation has a solution near $y_0$. For direct verification use the straightening theorem to transform the equation into
\[
\frac{\partial \sigma^*}{\partial x_1}=\lambda^*\sigma^*.
\]
This is a linear ordinary differential equation (with parameters $x_2,\ldots,x_n$), for which the existence of nontrivial solutions is obvious.

For part (b), note that $X_hX_f-X_fX_h=\lambda X_f$, hence $X_hX_f(\psi)=\lambda X_f(\psi)$ by Proposition \ref{liebproperties}(a), and then use Proposition \ref{liebproperties}(c).
\end{proof}
This result is not entirely satisfactory since it excludes stationary points, thus interesting local dynamics. But
with Proposition \ref{symvsorbsymprop} and Theorem \ref{symredthm}, the proof of the following reduction theorem is obvious.
\begin{theorem}\label{orbsymredthm}
Let $f,\,h\in\mathcal{A}(U)$ and $\lambda\in A(U)$ with $\left[h,\,f\right]=\lambda f$. Moreover let $\emptyset \neq W\subseteq \mathbb K^n$ be open and $\Psi:W\to U$ analytic and solution-preserving from $\dot{x} = \begin{pmatrix}
1 \\ 0 \\ \vdots \\ 0     \end{pmatrix}$ to $\dot{x} = h(x)$, with $D\Psi(x)$ invertible for all $x$. Define
\[
\widehat f(x):= D\Psi(x)^{-1}f(\Psi(x)).
\]
Then $\Psi$ is solution-preserving from $\dot{x} =\widehat f(x)$ to $\dot x = f(x)$, and moreover there exists an analytic $\rho$ on an open and dense subset of $W$ such that
\[
\widehat f(x) =\rho(x)\, \begin{pmatrix}
                          f_1^*(x_2,\ldots, x_n)  \\
			  \vdots    \\
			  f_n^*(x_2,\ldots, x_n)
                         \end{pmatrix}.
\]
\end{theorem}
According to the proof of  Proposition \ref{proporbequiv}(a), solutions of $\dot x=\widehat f(x)$ can be obtained from solutions of $\dot x=\begin{pmatrix}
                          f_1^*(x_2,\ldots, x_n)  \\
			  \vdots    \\
			  f_n^*(x_2,\ldots, x_n)
                         \end{pmatrix}$
with one further quadrature. In particular, for dimension $n=2$ the existence of a nontrivial infinitesimal orbital symmetry allows to solve $\dot x=f(x)$ by quadratures alone.

We note that Remark \ref{revengrem}  (Lie's ``reverse engineering'')  carries over to the orbital symmetry case, and indeed this was the situation that Lie focussed on. Moreover, Proposition \ref{redbyinv} carries over to the orbital symmetry case in an obvious way, with the caveat that orbital reducibility to dimension one does not yield useful information.
\begin{example}{\em We continue Example \ref{diagonalex2d} for the orbital symmetry case:
According to Theorem \ref{orbsymredthm}, every vector field of the form
\[
\widehat f(x)=\sigma(x)\begin{pmatrix}x_1f_1^*(x_1^{-\beta}x_2)\\ \beta x_2f_1^*(x_1^{-\beta}x_2)+x_1^\beta f_2^*(x_1^{-\beta}x_2)\end{pmatrix}
\]
admits the infinitesimal orbital symmetry $h$, with $h(x) = \begin{pmatrix}
                              x_1  \\ \beta\; x_2
                               \end{pmatrix}$.}
\end{example}
In dimension two, there is a different way to utilize infinitesimal (orbital) symmetries.

\begin{theorem}\label{ifacthm}
Let $h,f\in\mathcal{A}(U)$, and $\varphi(x):= \det(f(x),h(x))$ not identically zero.   \\
If $[h,f] = \lambda f$ for some $\lambda\in A(\widehat U)$ (with $\emptyset \not=\widehat U\subseteq U$), then
\[
X_f(\phi)={\rm div}\,f\cdot \phi,
\]
with ${\rm div}\,f={\rm tr}\,Df$. Thus $\phi^{-1}$ is an integrating factor of $\dot x=f(x)$.
\end{theorem}
\begin{proof} By multilinearity of the determinant, one has
\[
D\phi(x)y= \det\left(Df(x)y,h(x)\right)+ \det\left(f(x),Dh(x)y\right),
\]
and therefore
\[
\begin{array}{rcccl}
X_f(\phi)(x)&=& \det\left(Df(x)f(x),h(x)\right)&+& \det\left(f(x),Dh(x)f(x)\right)\\
                   &=& \det\left(Df(x)f(x),h(x)\right)&+& \det\left(f(x),Df(x)h(x)+\lambda(x)f(x)\right)\\
                 &=& \det\left(Df(x)f(x),h(x)\right)&+& \det\left(f(x),Df(x)h(x)\right)\\
                 &=&{\rm tr}\,Df(x) \cdot \det\left(f(x),h(x)\right),& & 
\end{array}
\]
using the Lie bracket condition and the alternating property of the determinant.
\end{proof}
One can make use of this result as follows: A two dimensional system $\dot{x}=g(x)$ on $U\subseteq \mathbb K^2$ is called \emph{Hamiltonian} if
$
{\rm div}\,g = 0$; 
locally this is equivalent to
\[
g=\begin{pmatrix}
   -\partial \varrho/\partial x_2   \\  \partial \varrho/\partial x_1                      \end{pmatrix}\text{  for some  } \varrho\in A(U),
\]
and $\varrho$ is then a first integral of $\dot{x}=g(x)$. (Finding first integrals of planar Hamiltonian systems involves only quadratures.) Now $\sigma\in A(U)$ is an integrating factor of  \eqref{ode} if and only if $\sigma \,f$ is Hamiltonian.

\begin{remark}\label{ifacrem}{\em 
\begin{itemize}
\item The converse of Theorem \ref{ifacthm} also holds: If the inverse of $\det(f(x),\,h(x))$ is an integrating factor of \eqref{ode} then $[h,f] = \lambda f$  on some open $\widehat U\subseteq U$, $\lambda\in A(\widehat U)$.
\item The statement of the theorem extends to dimension $n>2$, with the same proof (up to greater writing effort): Let $h_1,\ldots,h_{n-1}\in\mathcal{A}(U)$ be such that on some open and nonempty $\widehat U\subseteq U$ one has $[h_i,f] = \lambda_i f$ for $\lambda_i\in A(\widehat U)$, $1\leq i\leq n-1$. Then 
\[
\phi(x):=\det\left(f(x),\,h_1(x),\ldots,h_{n-1}(x)\right)
\]
satisfies $X_f(\phi)={\rm div}\,f\cdot \phi$. (The inverse of $\phi$ is a {\em Jacobi multiplier} of the system.)
\end{itemize}
}
\end{remark}
\begin{example}\label{ifacex}{\em 
\begin{itemize}
\item Let $\alpha\in\mathbb K$, $\alpha\not=1$ and
\[
f(x)=\begin{pmatrix}x_1\\ \alpha x_2\end{pmatrix},\quad h(x)=\begin{pmatrix}x_1\\ x_2\end{pmatrix},\text{  with  } \left[h,f\right]=0.
\]
From $\phi(x)=(1-\alpha)x_1x_2$ one finds the integrating factor $(x_1x_2)^{-1}$ for $f$.
\item For
\[
f(x)=\begin{pmatrix}x_1^2-x_2^2\\ 2x_1x_2\end{pmatrix},\quad h(x)=\begin{pmatrix}x_1\\ x_2\end{pmatrix},\text{ with } \left[h,f\right]=f,
\]
one computes $\phi(x)=-\left(x_1^2x_2+x_2^3\right)$, and after a little work one finds the first integral $x_2^{-1}(x_1^2+x_2^2)$ for $\dot x=f(x)$.
\end{itemize}
}
\end{example}
Finally we look briefly at non-autonomous systems.
\begin{remark}\label{nonautrem}{\em
The (seemingly more general) notion of symmetries of non-autonomous differential equations \eqref{nonaut} blends in naturally with the discussion of orbital symmetries. One only has to employ the trick to ``autonomize'' the system into the form \eqref{autnonaut}. Then symmetries of the nonautonomous system stand in one-to-one correspondence with orbital symmetries of the autonomization. First we note that in the non-autonomous context one considers locally invertible transformations of the type
\[
\begin{pmatrix}t\\ x\end{pmatrix}\mapsto \begin{pmatrix} \rho(t,x)\\ R(t,x)\end{pmatrix} 
\]
involving both independent and dependent variables. The obvious notion of a symmetry of \eqref{nonaut} is that it sends solutions to solutions. We now justify the correspondence to orbital symmetries of the autonomized system \eqref{autnonaut}. 
\begin{enumerate}[(i)]
\item The solutions $z(t)$ of \eqref{nonaut} stand in $1-1$ corespondence with solutions $\begin{pmatrix}t\\ z(t)\end{pmatrix}$ of \eqref{autnonaut}.
\item The transformation above sends solutions of \eqref{autnonaut} to solutions of some system
\[
\begin{array}{rcl}
\dot\xi&=&\gamma(\xi,x)\\
\dot x&=&g(\xi,x)
\end{array},\text{  orbitally equivalent to   }\begin{array}{rcl}
\dot\xi&=&1 \\
\dot x&=&\gamma(\xi,x)^{-1}g(\xi,x)
\end{array}.
\]
\item By the correspondence noted above, the transformation sends solutions of $\dot x=q(t,x)$ to solutions of $\dot x=\gamma(t,x)^{-1}g(t,x)$. Thus we have a symmetry of \eqref{nonaut} if and only if $\gamma^{-1}g=q$; equivalently 
\[
\begin{pmatrix} \gamma\\ g\end{pmatrix}=\gamma\cdot\begin{pmatrix}1\\ q\end{pmatrix},
\]
which is the condition for generic orbit equivalence.
\end{enumerate}
The standard approach is to discuss this matter as a special case in the wider (and more widely applicable) framework of jet spaces and prolongations; see e.g. Olver \cite{Ol 1}, Section 2.3. }
\end{remark}
\section{Intermezzo: The problem}
At this point (or probably earlier) the question may arise why not every ordinary differential equation (at least in dimension two) can be solved by quadratures: After all, given $f\in\mathcal{A}(U)$, one just has to find some $h$ such that $\left[h,\,f\right]=\lambda f$ (possibly with $\lambda =0$), reduce and repeat as long as necessary. But the problem lies in finding such a $h$. While there is the obvious trivial (but not useful) choice $h\in\mathbb K\cdot f$, carrying out the procedure e.g. from Theorem \ref{symredthm} would require to solve $\dot x=f(x)$ first.

 Perhaps surprisingly, finding explicit nontrivial infinitesimal (orbital) symmetries is complicated {\em just because} locally there exist too many of these: If $y_0$ is nonstationary for \eqref{ode}, then by the straightening theorem we may assume that $f=\begin{pmatrix}1\\ 0\\ \vdots\\ 0\end{pmatrix}$, thus $\left[h,\,f\right]=0$ if and only if $h$ depends only on $x_2,\ldots,x_n$. But due to this abundance there are too few restrictions, and this makes any ansatz to explicitly compute a commuting vector field $h$ for a general $f\in\mathcal{A}(U)$ (e.g. by successive determination of power series coefficients) unfeasible. Here lies the explanation for the fact that many textbooks and monographs on symmetries of differential equations hardly discuss first order ordinary differential equations. (Stephani \cite{Ste} even calls this class ``exceptional''.) Things are more interesting locally near stationary points, and we will get back to this later. Moreover one obtains an algorithmic approach for (a suitably restricted class of) symmetries of higher order equations. We discuss a simple instance of the latter in the next section.
\section{Symmetries of second order equations}
There is a vast literature on symmetries of higher order equations, starting with Lie's own work, and usually the focus is on non-autonomous second order equations. Here we just will give an impression of the matter, and will discuss only autonomous equations of second order, following work of Loos \cite{Lo 1,Lo 2}.
Thus we consider
\begin{equation}\label{secord}
 \ddot{x}=h(x,\dot{x})\quad \text{ on } \quad \widetilde U:=U\times \mathbb K^n,
\end{equation}
with corresponding first order system
\begin{equation}\label{secone}
\begin{array}{rcl}
 \dot x & =& y   \\
\dot{y} & = &h(x,y).
\end{array}
\end{equation}
We focus attention on symmetries of these equations of a special type, viz.\ those that are induced by a map 
\begin{equation}\label{pointmapeq}
\Psi:\,U\to U,\quad x\mapsto \Psi(x)
\end{equation}
with invertible Jacobian everywhere.  The induced map $\widehat\Psi:\,\widetilde U\to \widetilde U$ is then defined as
\[
   \widehat{\Psi}:\begin{pmatrix}           x \\ y
            \end{pmatrix}\mapsto\begin{pmatrix}
				\Psi(x)   \\
				D\Psi(x)y
				\end{pmatrix}\,.
\]
For motivation, note the following: If $\gamma(t)$ is a curve in $U$, then $\frac{d}{dt}\Psi(\gamma(t))=D\Psi(\gamma(t))\dot\gamma(t)$, hence
\[
  \widehat{\Psi}\begin{pmatrix}
             \gamma(t)  \\ \dot{\gamma}(t)
            \end{pmatrix} = \begin{pmatrix}
			\Psi(\gamma(t))  \\
			D\Psi(\gamma(t))\dot\gamma(t)
			\end{pmatrix}\,.
\]
(Again jet spaces and prolongations would provide the appropriate framework.) Note that $\widehat\Psi$ is also locally invertible everywhere.

As a first step we consider solution-preserving maps in the given scenario: Given a further equation $\ddot{x}=q(x,\dot x)$ on $U$, what are the conditions to ensure that $\Psi$ maps (parameterized) solutions of this equation to (parameterized) solutions of \eqref{secord}?
The answer is as follows:
\begin{lemma}\label{solpressecord}
\begin{enumerate}[(a)]
\item $\Psi$ sends solutions of $\ddot{x}=q(x,\dot x)$ to solutions of $\ddot{x}=h(x,\dot x)$ if and only if
\[
 D^2\Psi(x)(y,y)+D\Psi(x)q(x,y) = h(\Psi(x),D\Psi(x)y)
\]
for all $x,y$.
\item A local one-parameter group of transformations of $U$ with infinitesimal generator $g$ induces symmetries of  $\ddot{x}=h(x,\dot x)$ if and only if
\begin{equation}\label{secordinf}
 D^2 g(x)(y,y)+Dg(x)h(x,y)=D_1h(x,y)g(x)+D_2h(x,y)Dg(x)y.
\end{equation}
\end{enumerate}
\end{lemma}
\begin{proof}
 Concerning part (a), acccording to Lemma \ref{solprescrit} the necessary and sufficient condition is
\[
 D\widehat{\Psi}(x,y)\cdot \begin{pmatrix}
                       y \\ q(x,y)
                      \end{pmatrix} = \begin{pmatrix}
					D\Psi(x)y   \\
					h(\Psi(x),D\Psi(x)y)
					\end{pmatrix}\,,
\]
and the assertion follows with $D\widehat{\Psi}(x,y) = \begin{pmatrix}
                    D\Psi(x) & 0   \\
		    D^2\Psi(x)(y,\cdot) & D\Psi(x)
                   \end{pmatrix}$.
For part (b), set $q=h$, $\Psi=G(s,\cdot)$, differentiate with respect to $s$ and set $s=0$.
\end{proof}
The first observation here is that the vector space of all $g$ satisfying \eqref{secordinf} is finite dimensional (while the vector space of all infinitesimal symmetries of \eqref{secone} near a nonstationary point in $\widetilde U$ is infinite dimensional).
\begin{proposition}\label{dimbound}
If there exists $x_0\in U$ such that  $g(x_0)=0$ and $Dg(x_0)=0$, then $g=0$. Hence the vector space of all $g$ that satisfy \eqref{secordinf} has dimension $\leq n+n^2$, and for any $z\in U$, $g$ is uniquely determined by $g(z)$ and $Dg(z)$.
\end{proposition}
\begin{proof}
From \eqref{secordinf} one sees that $g(z)=0$ and $Dg(z)=0$ imply $D^2g(z)=0$. Now differentiate and use induction to show that $D^kg(z)=0$ for all $k$. 
\end{proof}
The second observation is that one obtains a quasi-algorithmic approach to compute infinitesimal symmetries of this special type, working degree by degree:
Expand 
\[
 h(x,y) = h^{(0)}(x)+h^{(1)}(x)(y)+h^{(2)}(x)(y,y)+h^{(3)}(x)(y,y,y)+\cdots,
\]
with $h^{(j)}(x)$ - at fixed $x$ -  a symmetric and multilinear map from $\underbrace{\mathbb K^n\times \cdots \times \mathbb K^n}_{j\;\,\text{terms}}$ to $\mathbb K^n$. Therefore
\begin{align*}
& D_1h(x,y)g(x)=Dh^{(0)}(x)g(x)+\left(Dh^{(1)}(x)g(x)\right)(y)+\left(Dh^{(2)}(x)g(x)\right)
(y,y)+\cdots,  \\
& D_2h(x,y)(w) = h^{(1)}(x) w+2h^{(2)}(x)(y,w)+3h^{(3)}(x)(y,y,w)+\cdots, 
\end{align*}
which allows a degree-by-degree evaluation.
\begin{itemize}
 \item Degree 0: $\quad Dg(x) h^{(0)}(x) = Dh^{(0)}(x)g(x)\quad$  (thus $[g,h^{(0)}]=0$\,); 
 \item Degree 1: $\quad Dg(x) h^{(1)}(x)y - \left(Dh^{(1)}(x)g(x)\right)y-h^{(1)}(x)Dg(x)y = 0$;
 \item Degree 2: $\quad D^2g(x)(y,y)+Dg(x)\left(h^{(2)}(x)(y,y)\right)
	-\left(Dh^{(2)}(x)g(x)\right)(y,y)-2h^{(2)}(x)(y,Dg(x)y)=0$;
 \item Degree $j\geq 3$: $\quad Dg(x) 
	h^{(j)}(x)(y,\ldots,y)-\left(Dh^{(j)}(x)g(x)\right)(y,\ldots,y)\\
-jh^{(j)}(x)(y,\ldots,y,Dg(x)y)=0$.
\end{itemize}
The degree two condition, which stands out because it does not fit the pattern of the other conditions, is amenable to a geometric interpretation: $h^{(2)}(x)$ defines a torsion-free affine connection on $\mathbb K^n$, with a corresponding differential equation $\ddot{x}=h^{(2)}(x)(\dot x,\dot x)$ for the geodesics. The definitions imply that the $g$ generates a local transformation group of automorphisms of this affine connection. We note one special consequence:
\begin{proposition}\label{affinesym}
If $h^{(2)}=0$, then $D^2g=0$, hence $g(x) = c+B(x)$ is given by an affine map.  \\
In particular, equations of the form $\ddot{x}=h^{(0)}(x)$ admit only affine maps as symmetries of type \eqref{secordinf}, and moreover $[g,h^{(0)}]=0$. \\
 \hspace*{2cm}\hfill  $\Box$ 
\end{proposition}
One can generalize this result to equations  $\ddot{x}=q(x,\dot x)$ that are, by way of Lemma \ref{solpressecord}, transformable to an equation $\ddot{x}=h(x,\dot x)$ with $h^{(2)}=0$. The degree two condition shows that this holds if and only if
\[
 q^{(2)}(x) = -D\Psi(x)^{-1}D^2\Psi(x).
\]
In dimension $n=1$ the condition reads $-\dfrac{\Psi''}{\Psi'}=q^{(2)}$, which can always be satisfied. We use this for a classification.
\begin{example}{\em
 Dimension 1:    Let the equation
 $\ddot{x}=h(x,\dot x)$ be given, w.l.o.g.\ with $h^{(2)}=0$. Then one has the following conditions for infinitesimal symmetries: 
\begin{itemize}
\item $g'\cdot h^{(0)}={h^{(0)}}' \cdot g$,   
\item ${h^{(1)}}' \cdot g=0\qquad$ (hence $h^{(1)}$ is constant whenever $g\neq 0$);   
\item $g''=0$ (hence $g$ is a polynomial of degree $\leq 1$);     
\item $(1-j)g'(x) h^{(j)}(x)={h^{(j)}}'(x)g(x)\quad$ for all $j\geq 3$.
\end{itemize}
We use an elementary fact: If  $p$,$q$ are nonzero functions of one variable then  $mq'p=p'q$ for some $m\in\mathbb Z$ if and only if $p=\gamma\cdot q^m$ for some constant $\gamma$, and vice versa.   \\
Therefore, if there exists $j>2$ such that $h^{(j)}\neq0$, then $g$ is determined by $h^{(j)}$ (up to a constant factor), and the space $\mathcal{L}$ of all vector fields satisfying the infinitesimal symmetry condition has dimension $\leq1$. The same holds whenever $h^{(0)}\neq 0$.   \\
In the case $\dim \mathcal{L}=1$ one finds
\begin{equation}\label{dimelloneeq}
 h(x,y)=\sum_{\substack{j\geq 0 \\ j\neq 2}} \gamma_j g(x)^{1-j}y^j,
\end{equation}
whenever there exists $l\not\in \{1,\,2\}$ such that $h^{(l)}\neq 0$.
The remaining case is that $h^{(j)}= 0$ for all $j\neq 1$ and $h^{(1)}$ constant. Thus the differential equation is of the form $\ddot{x}=\gamma\cdot x$, $\gamma\in\mathbb K$, and all polynomials of degree $\leq 1$ lie in $\mathcal{L}$.\\

The equations \eqref{dimelloneeq}, with one-dimensional space of symmetries, can be solved by quadratures: Essentially one has either $g=1$ and $h(x,y)=\sum \gamma_j y^j$ (thus $\ddot{x}=r(\dot x)$ with some function $r$; a first order equation for $\dot x$), or $g=x+\beta$ and $\ddot{x}=(x+\beta)\,r\left(\dfrac{\dot x}{x+\beta}\right)$ with some function $r$.  \\
In the latter case $z:=\dfrac{\dot x}{x+\beta}$ satisfies $\dot{z}=\dfrac{(x+\beta)\ddot{x}-\dot x^2}{(x+\beta)^2}=r(z)-z^2$; again we are left with quadratures.}
\end{example}
\begin{example}{\em
 Motion in a central force field (Dimension 2):
Here we have the equation $\ddot{x}=r(x_1^2+x_2^2)\cdot x$, with $r$ a nonzero function of one variable. Due to $h=h^{(0)}$, infinitesimal symmetries of type \eqref{secordinf} must be of the form $g(x)=B(x)+c$ ($B$ linear, $c$ constant).    \\
Evaluate $[g,h^{(0)}]=0$ to obtain, with $\varphi_1(x):=x_1^2+x_2^2$:
\[
 \left(r^\prime(\varphi_1(x))\cdot D\varphi_1(x)(Bx+c)\right)\cdot x+r(\varphi_1(x))\cdot c=0.
\]
The assumption $c\not=0$ and generic linear independence of $x$ and $c$ yields the contradiction $r=0$; thus necessarily $c=0$ and $D\varphi_1(x)Bx=0$. The latter condition forces (w.l.o.g.) that $B=\begin{pmatrix}
                                                0 & -1 \\ 1 & 0
						\end{pmatrix}$.   \\
The associated first order system $\dot{\binom{x}{y}}=\hat{f}(x,y)$, in detailed notation
\begin{align*}
 & \dot x=y   \\
 & \dot{y}= \rho(\varphi_1(y))x,
\end{align*}
therefore admits a one-parameter symmetry group with infinitesimal generator $\hat{B}=\begin{pmatrix}
                                                             B & 0  \\  0 & B
                                                                   \end{pmatrix}$.

There exists a reduction via  Proposition \ref{redbyinv}:
Defining
\[
  \varphi_2(x) = y_1^2+y_2^2,\quad \varphi_3(x)=x_1y_1+x_2y_2,\quad \varphi_4(x)=x_1y_2-x_2y_1,
\]
the relations
\[
 X_{\hat{f}}(\varphi_1)=2\varphi_3,\quad X_{\hat{f}}(\varphi_2)=2r(\varphi_1)\cdot \varphi_3,\quad X_{\hat{f}}(\varphi_3) = \varphi_2+r(\varphi_1)\varphi_1,\quad 
 X_{\hat{f}}(\varphi_4)=0
\]
show that $\Phi=\begin{pmatrix}\phi_1\\ \vdots\\ \phi_4\end{pmatrix}$ sends solutions of $\dot{\binom{x}{y}}=\hat{f}(x,y)$ to solutions of 
\[
\begin{array}{rcl}
\dot w_1&=& 2 w_3  \\
\dot w_2&=& 2 r(w_1) w_3  \\
\dot w_3&=& w_2+ r(w_1) w_1  \\
\dot w_4&=& 0
\end{array}
\]
The last equation directly yields a conservation law (angular momentum), and combining the first two equations leads to $dw_2/dw_1=r(w_1)$, which yields a first integral $w_2-R(w_1)$, with $R^\prime =r$ (conservation of energy). Moreover the relation $\varphi_1\varphi_2=\varphi_3^2+\varphi_4^2$ implies that only the (invariant) subvariety given by $w_1w_2-w_3^2-w_4^2=0$ is of interest.

In Chapter 4 we will see how the $\phi_j$ may be found.
}
\end{example}
\section{Further remarks and notes}
\begin{itemize}
\item 
From the abundant literature on symmetries of differential equations we mention just a few books and monographs. The classical book \cite{Li 1} by Sophus Lie is still available and readable (albeit in German), as is a later account by Engel and Faber \cite{EF}. As for (relatively) modern accounts, Olver's monograph \cite{Ol 1} stands out, but there are many other interesting  presentations from various perspectives, including Bluman and Kumei \cite{BK}, Stephani \cite{Ste} and Gaeta \cite{Ga}. For second order equations in Mechanics (a highly relevant application field) see also Marsden and Ratiu \cite{MR}. Concerning explicitly solvable differential equations, the classical book by Ince \cite{In} contains a large list (and also a short introduction to symmetry methods), and nowadays there exist many software packages (general and specialized) which are devoted to this field.
\item Concerning an extension to smooth vector fields and functions, one may roughly say that everything in Section 2.1 continues to hold (in some cases different proofs are necessary), mutatis mutandis, while everything in Section 2.2 becomes problematic. Locally, near non-stationary points, the results in this section can be salvaged, but (for instance) Remark \ref{orbpresrem} cannot be extended globally, and Proposition \ref{orbsymcharprop} is no longer true even locally. On the other hand, Loos \cite{Lo 1,Lo 2} actually proved his results for smooth second order equations, and considered the global setting. We mention here that in most accounts for symmetries of second order equations -- including Lie's own \cite{Li 1} -- one starts from non-autonomous equations, which yield somewhat different symmetry conditions that are also amenable to an algorithmic approach.
\end{itemize}
\chapter{Multiparameter symmetries and invariant sets}
We keep the notions and notation from Chapter 1. Moreover we introduce $M(U)$, the set of all {\em meromorphic} functions on $U$, i.e. functions that can (locally) be expressed as quotients of analytic functions. Since $U$ is connected, $M(U)$ is a field due to the identity theorem. 

In the present chapter we discuss the situation when ``more than one''  infinitesimal (orbital) symmetry of $\dot x=f(x)$ is given (more precisely, the corresponding vector space has dimension $>1$). The basic strategy will be to reduce by common first integrals of the symmetries.
\section{Structure of infinitesimal symmetries and reduction}
\begin{definition} \label{cendef} Let $f\in\mathcal {A}(U)$, and $\emptyset\not=U^*\subseteq U$ open and connected. 
\begin{enumerate}[(i)]
\item We call
\[
\mathcal{C}_{U^*}(f):=\left\{ h\in\mathcal{A}(U^*);\,\left[h,\,f\right]=0\right\}
\]
the {\em centralizer} of $f$ in $\mathcal{A}(U^*)$, and we call
\[
\mathcal{N}_{U^*}(f):=\left\{ h\in\mathcal{A}(U^*);\,\left[h,\,f\right]=\lambda \,f \text{  for some  }\lambda\in A(U^*)\right\}
\]
the {\em normalizer} of $f$ in $\mathcal{A}(U^*)$.
\item For given $h_1,\ldots, h_r\in\mathcal{A}(U^*)$ we denote by 
\[
I_{U^*}(h_1,\ldots,h_r)
\]
the set of constant functions and common first integrals of the $h_i$ on $U^*$.
\end{enumerate}
\end{definition}
The naming ``normalizer'' slightly abuses language. To some extent, Proposition \ref{symvsorbsymprop} reduces investigation of the normalizer (and reduction) to investigation of the centralizer, but this is not entirely satisfactory since these results are local and do not apply to any neighborhood of a stationary point.

We collect some elementary properties first:
\begin{lemma}\label{invollem} Let the notation be as in Definition \ref{cendef}. Then:
\begin{enumerate}[(a)]
\item $I_{U^*}(h_1,\ldots,h_r)$ is a subalgebra of $A(U^*)$.
\item $\mathcal{C}_{U^*}(f)$ and $\mathcal{N}_{U^*}(f)$ are Lie algebras and also modules over $I_{U^*}(f)$.
\item The set
\[
\left\{ g\in \mathcal{A}(U^*);\,X_g(\psi)=0\text{  for all  }\psi\in I_{U^*}(h_1,\ldots,h_r)\right\}
\]
is a Lie algebra as well as a module over $A(U^*)$.
\end{enumerate}
\end{lemma}
\begin{proof}
Part (a) follows from linearity and the product rule for derivations, the nontrivial Lie algebra properties in part (b) are consequences of the Jacobi identity and Proposition \ref{liebproperties}(c), and the latter also shows the module property. As for part (c), see Propsition \ref{liebproperties}(a).
\end{proof}
Thus, when dealing with common first integrals of two or more vector fields one has to take into account that these are also first integrals of their commutators. Taking a first step toward a fundamental existence result, we state:
\begin{proposition}\label{involprop} Let $U^*\subseteq U$ be nonempty, open and connected, and let $h_1,\ldots, h_s\in\mathcal{C}_{U^*}(f)$ be linearly independent over the field $M(U^*)$ of meromorphic functions. Then:
\begin{enumerate}[(a)]
\item If $h=\sum\beta_i h_i\in\mathcal{C}_{U^*}(f)$ with $\beta_i\in M(U^*)$, then all  $X_f(\beta_i)=0$. Thus every $\beta_i$ is constant or a first integral of \eqref{ode}.
\item There exist $r$, with $s\leq r\leq n$ and $h_j$, $s+1\leq j\leq r$, such that $h_1,\ldots, h_r$ are linearly independent over $M(U^*)$, and moreover
\[
\left[h_i,\,h_k\right]=\sum_\ell \mu_{ik\ell}\,h_\ell,\quad 1\leq i,k,\ell\leq r
\]
with each $\mu_{ik\ell}$ constant or a first integral of \eqref{ode}.
\item The $\mu_{ik\ell}$ are analytic on
\[
\widehat U:=\left\{y\in U^*;\,h_1(y),\ldots, h_r(y)\text{  are linearly independent in  }\mathbb K^n \right\}.
\]
\end{enumerate}
\end{proposition}
\begin{proof}
Part (a), and the second assertion of part (b), follow with Propsition \ref{liebproperties}(c), which implies
\[
0=\left[f,\,\sum\beta_i h_i\right]=\sum_iX_f(\beta_i)h_i+\sum_i\beta_i \left[f,\,h_i\right]=\sum_iX_f(\beta_i)h_i.
\]
The first assertion of part (b) is obtained by successively adjoining commutators as long as they still yield a linearly independent set over $M(U^*)$. (This implies $r\leq n$.) Finally, part (c) follows with Cramer's rule.
\end{proof}
\begin{remark}{\em 
In the setting of parts (b) and (c) above we say that $h_1,\ldots,h_r$ are {\em in involution} on $\widehat U$, and we extend this notion to any open set $V\subseteq U^*$ where all the $h_i$ and all the $\mu_{ik\ell}$ are analytic.}
\end{remark}

We illustrate with a small example how part (b) can be put to use.
\begin{example}{\em
Systems in $\mathbb R^3$ with \emph{so}$(3,\mathbb R)$-symmetry: For the basis elements 
\[
 h_1(x):=\begin{pmatrix}
          x_2 \\ -x_1 \\ 0
         \end{pmatrix},\qquad h_2(x):=\begin{pmatrix}
         			 x_3 \\ 0 \\ -x_1
         			\end{pmatrix}, \qquad  h_3(x):= \begin{pmatrix}
         			 			0 \\ x_3 \\ -x_2
         						\end{pmatrix}
\]
of the Lie algebra \emph{so}$(3,\mathbb R)$ one finds 
\[
h_3=\-\frac{x_3}{x_1}\cdot h_1 + \frac{x_2}{x_1}\cdot h_2.
\]
Thus any vector field $f$ admitting the infinitesimal symmetries $h_1,h_2,h_3$ necessarily admits the first integrals $\frac{x_2}{x_1}$ and $\frac{x_3}{x_1}$. This implies $f(x) = \mu(x)\cdot x$ for suitable $\mu$, and now Proposition \ref{liebproperties}(c) shows that $\mu$ is a common first integral of the $h_i$ (hence a function of $x_1^2+x_2^2+x_3^2$).

The property underlying this particular effect is that the dimension of  the group SO$(3,\mathbb R)$ is greater than the generic dimension of the group orbits.
}
\end{example}
So far we have not yet established the existence of common first integrals, except for the case of a single vector field (see Lemma \ref{locfirstint}). Existence is (locally) taken care of by the next theorem, which is due to Frobenius.
\begin{theorem}\label{frobenius} Let $U^*\subseteq U$ be nonempty, open and connected, and let $1\leq r<n$ and $h_1,\ldots, h_r\in\mathcal{C}_{U^*}(f)$ be linearly independent over $M(U^*)$. Moreover assume that
\[
\left[h_i,\,h_k\right]=\sum_\ell \mu_{ik\ell}\,h_\ell,\quad 1\leq i,k,\ell\leq r,\text{  with  } \mu_{ik\ell}\in M(U^*),
\]
and define
\[
\widehat U:=\left\{y\in U^*;\,h_1(y),\ldots, h_r(y)\text{  are linearly independent in  }\mathbb K^n \right\}.
\]
Then for every point of $\widehat U$ there exist a neighborhood $\widehat V$ and functionally independent $\varphi_1,\ldots,\varphi_{n-r}\in A(\widehat V)$ such that:
\begin{enumerate}[(a)]
 \item Every $\varphi_j$ is a common first integral of  $h_1,\ldots,h_r$.
 \item For every common first integral $\psi$  of the $h_i$ on $\widehat V$ there exists an analytic function $\sigma$ of $r$ variables such that $\psi=\sigma(\varphi_1,\ldots,\varphi_{n-r})$.
\end{enumerate}
\end{theorem}
\begin{proof}
\begin{enumerate}[(i)] 
\item We first derive the following auxiliary result: There exist an open and dense $V \subseteq \widehat U$ and analytic $g_1, \ldots,g_r$ on $V$ such that 
\[
A(V) g_1+\cdots +A(V) g_r=A(V) h_1+\cdots +A(V) h_r
\]
and furthermore $[g_i,g_j]=0$ for all $i,j$. \\
To prove this we may assume that the first $r$ rows of the matrix 
\[
(h_1(y),\ldots,h_r(y)) \in \mathbb K^{(n,r)}
\]
 are linearly independent at every point $y\in V$. By Cramer's rule there exist $\alpha_{ij}\in A(V)$ such that
\[
 g_i: = \sum^r_{j=1}\alpha_{ij}h_j= \begin{pmatrix}
     0 \\ \vdots \\ 0 \\ 1 \\ 0 \\ \vdots\\ 0 \\ \ast \\ \vdots \\ \ast
     \end{pmatrix} \begin{matrix}
                   \left\}\begin{matrix}
			 \\ \\ r \\ \\ \\ 
		    \end{matrix}\right. \\ \\ \\ \\ 
           \end{matrix}
\;;\quad \text{with the } 1\text{  in position  }\;\# i,
\]
and $(\alpha_{ij})$ is an invertible $r\times r$ matrix, with all entries of the inverse matrix also in $A(V)$. Therefore all $h_j\in A(V) g_1 +\cdots +A(V) g_r$. Since  $h_1,\ldots,h_r$ are in involution on $V$, so are $g_1,\ldots,g_r$ (with coefficients in $M(V)$). Evaluation of the Lie brackets now shows
\[
 [g_i,g_j] = \begin{pmatrix}
     0 \\ \vdots \\ 0 \\ \ast \\ \vdots \\ \ast
     \end{pmatrix} \begin{matrix}
                   \left\}\begin{matrix}
			\\ r \\ \\ 
		    \end{matrix}\right. \\ \\ \\ \\ 
           \end{matrix}\;, \quad \text{which forces}\;\;[g_i,g_j]=0.
\]
\item According to (i), the common first integrals of $h_1,\ldots,h_r$ resp.\ $g_1,\ldots,g_r$ coincide. To prove the theorem, one may therefore assume $[h_i,h_j]=0$ for all $i,j$. The proof is by induction on the dimension $n$, with the start $n=2$ (and necessarily $r=1$) known from Lemma \ref{locfirstint}.

For the induction step, by Proposition \ref{liecompatible} and the straightening theorem (on some $V\subseteq \widehat V$) one may assume that 
\[h_1 = \begin{pmatrix}
                                                        1 \\ 0 \\ \vdots \\ 0
                                                       \end{pmatrix}.
\]
Now $[h_1,h_i]=0$ implies that $h_2,\ldots,h_r$ are functions of $x_2,\ldots,x_n$ alone. Setting
\[
 h_i(x) = \begin{pmatrix}
                          \mu_i(x_2,\ldots,x_n)   \\
			  \hat{h}_i(x_2,\ldots,x_n)
                         \end{pmatrix}\,,
 \;\;\text{and}\;\;\widetilde{h}_i(x):=h_i(x)-\mu_i(x_2,\ldots,x_n) h_1(x),
\]
one verifies directly that $[\widetilde{h}_i,\widetilde{h}_j] = [h_1,\widetilde{h}_i]=0$ for all $i,j$. By the induction hypothesis, $\hat{h}_2,\ldots,\hat{h}_r$ admit $(n-1)-(r-1) = n-r$ independent common first integrals on $V\cap \mathbb K^{n-1}$ (possibly after replacing $V$ by an open subset), which satisfy both assertions. These common first integrals, considered as functions of $x_1,\ldots,x_n$, are also first integrals of $h_1$ and thus of all $h_i$. 
\end{enumerate}
\end{proof}
The reduction theorem now is a direct application of  Proposition \ref{redbyinv}.
\begin{theorem}\label{frobred}
Let $f\in \mathcal{A}(U)$ and let $h_1,\ldots,h_r$ be analytic on a nonempty open set $\widehat U \subseteq U$ such that $[h_i,f]=\lambda_i f$ with $\lambda_i\in A(\widehat U)$, $1\leq i\leq r$. Furthermore assume that the $h_i$ are in involution on $\widehat U$, with $h_1(y),\ldots, h_r(y)$ linearly independent for all $y\in \widehat U$. Finally, let $\emptyset \neq V\subseteq \widehat U$ be open, and $\varphi_1,\ldots, \varphi_{n-r}$ common first integrals of the $h_i$ on $V$ which satisfy the conclusion of Theorem \ref{frobenius}.
\begin{enumerate}[(a)]
 \item If all $[h_i,f]=0$ then $\Phi=\begin{pmatrix}
                                            \varphi_1 \\ \vdots \\ \varphi_{n-r}
                                           \end{pmatrix}:V\to \mathbb K^{n-r}$
sends solutions of $\dot{x}=f(x)$  to solutions of a system $\dot{x}=g(x)$ on some open subset of $\mathbb K^{n-r}$.  
\item If the $\lambda_i$ are not necessarily zero, then the map $\Phi$ is orbit-preserving from $\dot{x}=f(x)$ (restricted to some open and dense subset of $V$) to solutions of a system $\dot{x}=g(x)$ on some open subset of $\mathbb K^{n-r}$.  
\end{enumerate}
\end{theorem}
\begin{proof}
Part (a) follows from Proposition \ref{redbyinv}. Part (b) is obvious when every common first integral of the $h_i$ is also a first integral of $f$ (and thus $g=0$), otherwise it is a consequence of part (a) and Proposition \ref{symvsorbsymprop}.
\end{proof}
One might initially hope for a step-by-step reduction in the setting of Theorem \ref{frobred}, with each step decreasing the dimension by one. This is, however, not the case in general. The following result (essentially due to Bianchi \cite{Bi}) shows what conditions are necessary for a first step.
\begin{proposition}\label{bianchiprop}
With the hypotheses as in Theorem \ref{frobred},  assume furthermore $[h_j,h_1]\in A(\widehat U)h_1$ for all $j$. Let $y\in\widehat U$ be nonstationary for $h_1$, $V$ a suitable open neighborhood of $y$ and $\Psi:\, V\to\mathbb K^n$ locally invertible and solution preserving from $\dot{x} = \begin{pmatrix}
						1 \\ 0 \\ \vdots \\ 0  
                                                 \end{pmatrix}$ to $\dot{x} = h_1(x)$. 
On $V$ define $f^*(x):=D\Psi(x)^{-1}f(\Psi(x))$ and $h_j^*(x):=D\Psi(x)^{-1}h_j(\Psi(x))$, $2\leq j\leq r$.
\begin{enumerate}[(a)]
\item
Then $\Psi$ is solution-preserving from $\dot{x} = h^*_j(x)$ into $\dot{x}= h_j(x)$, and from 
$\dot{x}=f^*(x)$ into $\dot{x}=f(x)$.
\item Moreover
\[
 f^*(x) = \mu(x) \begin{pmatrix}
                  \varphi(x_2,\ldots,x_n)  \\
		  \hat{f}(x_2,\ldots,x_n)
                 \end{pmatrix}\,;\qquad 
h_j^*(x) =  \begin{pmatrix}
                  \eta_j(x_1,\ldots, x_n)  \\
		  \hat{h}_j(x_2,\ldots,x_n)
                 \end{pmatrix}\,
\]
with suitable analytic functions and vector fields.
\item The $\hat h_j$ are in involution on a suitable open $\emptyset\neq V^*\subseteq \mathbb K^{n-1}$, and $[\hat{h}_j,\hat{f}]= \hat{\lambda}_j\cdot \hat{f}$, $2\leq j\leq r$, with suitable $\hat\lambda_j\in A(V^*)$.
\end{enumerate}
\end{proposition}
\begin{proof} Part (a) is a direct consequence of the definitions. The first statement of part (b) follows from Proposition \ref{symvsorbsymprop}, while the second can be seen from
\[
 \frac{\partial h_j^*}{\partial x_1} = \mu_j\cdot h_1^* = \begin{pmatrix}
							\mu_j \\ 0 \\ \vdots \\ 0
                                                         \end{pmatrix}.
\]
For part (c), set $\widetilde{f}:= \begin{pmatrix}
                      \varphi \\ \hat{f}
                     \end{pmatrix}$. Then (with Proposition \ref{liebproperties}) one has $[h_j^*,\widetilde{f}] = \mu_j^*\cdot \widetilde{f}$ for suitable $\mu_j^*$, while
\[
 [h_j^*,h_\ell^*]\in \sum_i\, A(V)\cdot h_i^*\,,\quad 2\leq j,\ell\leq r
\]
continues to hold.
With
\[
 [h_j^*,f^*]= \begin{pmatrix}
                  \ast \\ [\hat{h}_j,\hat{f}]
                  \end{pmatrix}\,, \qquad [h_j^*, h_\ell^*] = \begin{pmatrix}\ast \\ [\hat{h}_j,\hat{h}_\ell ]\end{pmatrix}
\]
the assertion follows.
\end{proof}
\begin{example}{\em
Let $\varrho(x):= 2x_1x_3-x_2^2$, and consider the differential equation
\begin{align*}
 \dot{x}_1 & = x_2x_3 + \varrho\cdot x_1   \\
 \dot{x}_2 & = x_3^2 + \varrho \cdot x_2  \hspace{2cm} (\,\text{briefly}\;\; \dot{x}=f(x)\,).  \\
 \dot{x}_3 & = \phantom{x_3^2 +}\varrho\cdot x_3 
\end{align*}
This equation admits the infinitesimal symmetries
\[
h_1(x) = \begin{pmatrix}
                                                x_2 \\ x_3 \\ 0
                                                 \end{pmatrix}, \quad h_2(x) =
\begin{pmatrix}
 x_1 \\ 0 \\ -x_3
\end{pmatrix},\quad\text{with  }[h_1,h_2] = 2h_1.
\]
We straighten $h_1$:
For  $\Phi(x):= \begin{pmatrix}
               x_2/x_3 \\ x_3 \\ \varrho(x)
               \end{pmatrix}$ one has $D\Phi(x) h_i(x) = h_i^*(\Phi(x))$, with
\[ h_1^* = {\begin{pmatrix}
	  1 \\ 0 \\ 0
	 \end{pmatrix}},\quad 
h_2^* = -\begin{pmatrix}
	 x_1 \\ x_2 \\ 0
	\end{pmatrix},\text{ and  } D\Phi(x) f(x) = f^*(\Phi(x)), \quad f^*(x) = \begin{pmatrix}
	x_2 \\ x_2 x_3 \\ x^2_3
	\end{pmatrix}.
\]
Since the hypotheses of Proposition \ref{bianchiprop} are satisfied, we can proceed with
\[
\hat h_2=\begin{pmatrix}-x_2\\ 0\end{pmatrix}\text{  and  } \hat f=\begin{pmatrix}
	 x_2 x_3 \\ x^2_3
	\end{pmatrix}.
\]
For this particular example one sees directly that the solution of $\dot x=\hat f(x)$ reduces to quadratures. Proceeding by straightening $\hat h_2$ is possible but unnecessary.}
\end{example}
\section{Invariant sets}
In this section we discuss invariant sets of an analytic first-order equation \eqref{ode} that are in some way related to, or induced from, symmetries or orbital symmetries.
We start with first integrals, slightly generalizing Proposition \ref{redbyinv}(a). The proof is a straightforward application of Proposition \ref{liebproperties}.
\begin{proposition}\label{morefirstints} Let $f\in \mathcal{A}(U)$, $\emptyset\not=U^*\subseteq U$ open and $h_1,\ldots,h_r\in \mathcal{A}(U^*)$ such that $\left[h_i,\,f\right]=\lambda_i\,f$ with $\lambda_i\in A(U^*)$, $1\leq i\leq r$. Then:
\begin{enumerate}[(a)]
\item For any $\phi\in A(U^*)$ with $X_f(\phi)=0$ one also has $X_f\left(X_{h_i}(\phi)\right)=0$ for $1\leq i\leq r$.
\item For any $\gamma\in A(U^*)$ such that all $X_{h_i}(\gamma)=0$ one has $X_{h_i}\left(X_{f}(\gamma)\right)=-\lambda_iX_f(\gamma)$ for $1\leq i\leq r$.
\end{enumerate}
\end{proposition}
\begin{remark}{\em
While the proof of Frobenius' Theorem, and hence the reduction procedure from Theorem \ref{frobred} is constructive only in a rather limited sense, one may use Proposition \ref{morefirstints} in a semi-constructive way to find reduced equations when some partial information is available. We outline one such instance.
\begin{itemize}
\item Assume that all $\left[h_i, \,f\right]=0$ and that a common first integral $\gamma$ of the $h_i$ is given. Then all $X_f^\ell(\gamma)$ are common first integrals of the $h_i$ (or constant), and there exists a largest $k\geq 0$ such that 
$\gamma,\,X_f(\gamma),\ldots,X_f^k(\gamma)$ are functionally independent on $U^*$. There is an open and dense subset of $U^*$ such that in some neighborhood of any point one has 
\[
X_f^{k+1}(\gamma)=\sigma(\gamma,X_f(\gamma),\ldots,X_f^k(\gamma)),
\]
with some analytic function $\sigma$ of $k+1$ variables, and this induces a solution-preserving map from \eqref{ode} to the differential equation
\[
\begin{array}{rcl}
\dot x_0&=& x_1\\
              &\vdots & \\
\dot x_{k-1}&=& x_{k}\\
\dot x_k&=& \sigma(x_0,\ldots, x_k)
\end{array}
\]
\item The non-constructive ingredient in the previous argument is the function $\sigma$, whose existence is generally ensured only by the implicit function theorem. But one may replace the argument by an implicit version, which becomes constructive at least for polynomial vector fields and functions, and which we sketch now: Assume that $f$ is a polynomial vector field and $\gamma$ a polynomial. Then -- with $k$ as above -- there exists a polynomial $\rho\not=0$ in $k+2$ variables (which is in principle algorithmically accessible) such that 
\[
\rho(\gamma,X_f(\gamma),\ldots, X_f^{k+2}(\gamma))=0.
\]
The map 
\[
U^*\to\mathbb K^{k+2},\quad x\mapsto \begin{pmatrix}\gamma(x)\\ X_f(\gamma)(x)\\ \vdots \\X_f^{k+1}(\gamma)(x)\end{pmatrix}
\]
is then solution-preserving to some polynomial differential equation on $\mathbb K^{k+2}$ which admits the invariant set 
$Y$ defined by $\rho=0$. (And only the restriction to $Y$ is of interest.)
\end{itemize}
}
\end{remark}
We look at further invariant sets, first dealing with symmetries.
\begin{theorem}\label{fixinvthm} Let $f\in\mathcal{A}(U)$ and $\emptyset\not=U^*\subseteq U$ open.
\begin{enumerate}[(a)]
\item If $\Phi:U^*\to U$ is a symmetry of $\dot{x}= f(x)$, then the fixed point set of $\Phi$ is invariant for $\dot{x}=f(x)$.
\item Let $h\in \mathcal{A}(U^*)$ and $[h,f]=0$. Then the set  of stationary points of $h$ is invariant for $\dot{x}=f(x)$.
\item 
Let $h_1,\ldots,h_r\in\mathcal{A}(U^*)$ and  $[h_i,f]=0$ for all $i$. Then for all $s$, $0\leq s\leq r$ the set
\[
Z_s:= \{y\in U^*:\;{\rm rank}\,\left(h_1(y),\ldots,h_r(y)\right)\leq s\}
\]
is invariant for $\dot{x}=f(x)$.
\end{enumerate}
\end{theorem}
\begin{proof}
 (a) Since $\Phi$ is a symmetry, one has $\Phi(F(t,y)) = F(t,\Phi(y))$ for all $y\in U^*$ and all $t$ in the maximal existence interval for the restriction of \eqref{ode} to $U^*$. Thus $\Phi(y) = y$ implies $\Phi(F(t,y)) = F(t,y)$ for all $t$.  
To prove part (b), apply (a) to  $H_s:\,y\mapsto H(s,y)$. Then for every fixed $s$ near $0$ the fixed point set of $H_s$ is invariant for $\dot{x}=f(x)$, and so is the intersection of these fixed point sets, which is equal to the set of stationary points of $h$.\\ As for part (c), we actually prove a more detailed statement, w.l.o.g.\ assuming that $Z_s\not=Z_{s-1}$ in the argument.
Thus let $y_0\in Z_s$ and ${\rm rank}\,(h_1(y_0),\ldots,h_r(y_0))=s$; we may assume that $h_1(y_0),\ldots,h_s(y_0)$ are linearly independent. There exist  $\mu_{ij}\in\mathbb K$ such that $h_j(y_0)-\sum\limits^s_{i=1} \mu_{ij}h_i(y_0) = 0$ for  $s+1\leq j\leq r$; hence $y_0$ lies in the common zero set of the  $h_j-\sum \mu_{ij}h_i$, $s+1\leq j\leq r$. This set is invariant by (b), and clearly a subset of $Z_s$.
\end{proof}
\begin{remark}{\em
Note that the $h_i$ are not required to be in involution here. Moreover one may include $f$ among the $h_i$.}
\end{remark}
We discuss some applications and limitations of the theorem.
\begin{example} {\em 
\begin{itemize}
 \item Let $\mathbb K=\mathbb R$, $n>2$ and assume that $\dot{x}= f(x)$ admits every element of $SO(n,\mathbb R)$ as a symmetry. Then $f(x)= \mu(x_1^2+\cdots +x_n^2)\cdot x$ with a suitable function $\mu$ of one variable.  \\
Indeed, if $n$ is odd then every one-dimensional subspace of $\mathbb R^n$ is the fixed point space of some element of $SO(n,\mathbb R)$. Thus $f(y)$ and $y$ are linearly dependent in $\mathbb R^n$ for all $y$, which implies  $f(x)=\varrho(x)\cdot x$. Moreover $\varrho$ is analytic and a first integral for every element of the Lie algebra $so(n,\mathbb R)$; this implies the assertion.  \\
If $n>2$ is even then every two-dimensional subspace of $\mathbb R^n$ is the fixed point space of some element of $SO(n,\mathbb R)$, and every one-dimensional subspace is the intersection of two such fixed point spaces, hence invariant. The remainder of the argument works as above. 
\item In Example \ref{matrixriccex} we discussed homogeneous matrix Riccati equations $\dot x=q_a(x):=xax$ in $\mathbb K^{(n,n)}$, and saw that $\left[q_a,\,q_b\right]=0$ for all $a, \,b$. From Theorem \ref{fixinvthm} one sees now, for instance, that for any $c$ the set 
\[
Y_0=\{y;\, ycy=0\}
\]
is invariant for any equation $\dot x=q_a(x)$. This set includes, for example, all matrices with $y^2=0$ (taking $c$ as the unit matrix).
\item  If $\dot{x}=f(x)$ admits an isolated stationary point $z$, then $z$ is also stationary for any infinitesimal symmetry $h$, since $H(s,z)$ is stationary for \eqref{ode} for all $s$ near $0$: For $s\to 0$ one has $H(s,z)\to z$, and since $z$ is isolated, one has $H(s,z)=z$ for all $s$.
\item 
Theorem \ref{fixinvthm} is not generally true for (infinitesimal) orbital symmetries. For a counterexample, let
\[
f(x) = \begin{pmatrix}
                      1 \\ 0
               \end{pmatrix}, \quad h(x) =\begin{pmatrix}
                      			x_1 \\ x_2
               				\end{pmatrix}\quad \text{on}\;\; \mathbb K^2. 
\]
Then $[h,f]=-f$, hence $h$ is an infinitesimal orbital symmetry for $\dot x= f(x)$, but the set $\{0\}$  of stationary points of $h$ is not invariant for $f$. 
\end{itemize}
}
\end{example}
An appropriate version of Theorem \ref{fixinvthm} for infinitesimal orbital symmetries follows.
\begin{theorem}\label{lindepinvthm} Let $f\in\mathcal{A}(U)$ and $\emptyset\not=U^*\subseteq U$ open. Moreover, let
 $h_1,\ldots,h_r\in\mathcal{A}(U^*)$ such that $[h_i,f]=\lambda_if$ for suitable $\lambda_i\in A(U)$, $(1\leq i \leq r)$.
 Then for every $s\leq r$ the set
\[
 Y_s:=\{y\in U:\;{\rm rank}\,(f(y),h_1(y),\ldots,h_r(y))\leq s+1\}
\]
is invariant for $ \dot x = f(x)$. 
\end{theorem}
\begin{proof}
We may assume that $s=r$ in the following, by discarding some $h_i$ if necessary. (Note that invariance is a local property.)
\begin{enumerate}[(i)]
 \item  Linear algebra (reminder): Let $\nu_1,\ldots,\nu_n$ denote the coordinate functions on $\mathbb K^n$;  
thus $\nu_i\begin{pmatrix}
 		x_1 \\ \vdots \\ x_n
             \end{pmatrix}=x_i$; furthermore fix $q<n$. Every sequence $1\leq i_0<i_1<\cdots< i_q\leq n$ of positive integers defines a  multilinear and alternating map
\[
 \Delta_{i_0,\ldots,i_q}:(\mathbb K^n)^{q+1}\to \mathbb K, \quad(v_0,\ldots,v_q)\mapsto \det\left((\nu_{i_k}(v_l))_{k,l}\right)\,.
\]
By the universal property of the Grassmann algebra, every $(q+1)$-linear alternating map from $(\mathbb K^n)^{q+1}$ to  $\mathbb K$ is a linear combination of such $\Delta_{i_0,\ldots,i_q}$.  \\
Moreover consider a linear map $B:\mathbb K^n\to\mathbb K^n$, and fix $j_0<\cdots< j_q$. Then
\[
 \delta_B:\, (v_0,\ldots,v_q)\mapsto \sum^q_{i=0}\;\Delta_{j_0,\ldots,j_q}(v_0,\ldots,v_{i-1},Bv_i,v_{i+1},\ldots,v_q)
\]
is clearly multilinear and alternating, hence
\[
 \delta_B = \sum_{(i_0,\ldots,i_q)}\;\alpha_{i_0,\ldots,i_q}(B)\cdot\Delta_{i_0,\ldots,i_q}
\]
with suitable coefficients $\alpha_{i_0,\ldots,i_q}$. One verifies that $B\mapsto \alpha_{i_0,\ldots,i_q}(B)$ is linear.
\item $Y_r$ is the common zero set of all
\[
 \varrho_{i_0,\ldots,i_r}(x):=\Delta_{i_0,\ldots,i_r} (f(x),g_1(x),\ldots,g_r(x)).
\]
Fix  $j_0,\ldots,j_r$ and abbreviate $\Delta:=\Delta_{j_0,\ldots,j_r}$, $\varrho:=\varrho_{j_0,\ldots,j_r}$.  \\
Using the multilinearity of the determinant, we compute the Lie derivative of $\varrho$:
\[
\begin{array}{rcl}
 X_f(\varrho)(x)& =& D\varrho(x) f(x)  \\
& =&\; \Delta(Df(x)f(x),h_1(x),\ldots,h_r(x))\\
&&+\sum_{1\leq i\leq r}\Delta (f(x),h_1(x),\ldots,Dh_i(x)f(x),\ldots,h_r(x))   \\
& = &\; \Delta(Df(x)f(x),h_1(x),\ldots,h_r(x))\\
&&+\sum_{1\leq i\leq r}\Delta (f(x),h_1(x),\ldots,Df(x)h_i(x),\ldots,h_r(x))   \\
&& -\sum_{1\leq i\leq r}\Delta (f(x),h_1(x),\ldots,\lambda_i(x)f(x),\ldots,h_r(x))\,,
\end{array}
\]
using  $[h_i,f]=\lambda_if$.  \\
The last term on the right hand side vanishes, and by  (i) the remaining terms are linear combinations of suitable $\varrho_{j_0,\ldots,j_r}$ with analytic functions as coefficients. The assertion follows with the invariance criterion from Proposition \ref{invarcritprop}.
\end{enumerate}
\end{proof}
This result generalizes Theorem \ref{ifacthm} and Remark \ref{ifacrem}. As in Theorem \ref{fixinvthm} there is no requirement on the $h_i$ to be in involution.
\begin{example} {\em 
 In $\mathbb K^3$ consider the polynomial vector fields
\[
 f(x) = \begin{pmatrix}
         x_1^2-x_2x_3 \\
	   2x_1x_2   \\
	   2x_1x_3
        \end{pmatrix},\quad h_1(x) = \begin{pmatrix}
         			     x_1 \\
	 			    3x_2   \\
	  			    -x_3
       				 \end{pmatrix},\quad h_2(x)= \begin{pmatrix}
         			    			 x_1 x_2\\
	 			   			 x_2^2   \\
	  			    			-x_1^2
       				 			\end{pmatrix}\, ,
\]
with $[h_1,f]=f$ and $[h_2,f]=0$.  According to Theorem \ref{lindepinvthm}, the zero set of 
\[
 \varrho(x):= \det(f(x),g(x),h(x)) = -x_2(x_1^2+x_2x_3)^2
\]
is invariant for $\dot x=f(x)$. This yields the plane $x_2=0$ and the less obvious invariant cone $x_1^2+x_2x_3=0$.  \\
Applying the theorem to $f$ and $h_1$ yields invariance of the common zero set of 
\[
 \det\begin{pmatrix}
      x_1^2-x_2x_3 & x_1  \\
      2x_1x_2 & 3x_2
     \end{pmatrix} = (x_1^2-3x_2x_3)x_2,\quad \det\begin{pmatrix}
     				 x_1^2-x_2x_3 & x_1  \\
      				2x_1x_3 & -x_3
    				 \end{pmatrix} = (-3x_1^2+x_2x_3)x_3
\]
and
\[
\det\begin{pmatrix}
     	 2x_1x_2 & 3x_2  \\
      	 2x_1x_3 & -x_3
    	\end{pmatrix} = -8x_1x_2x_3.
\]
This set is the union of three straight lines given by $x_1=x_2=0$, $x_2=x_3=0$, and $x_3=x_1=0$, respectively.}

\end{example}
Finally we note that invariant sets may also be obtained from solution-preserving maps.
\begin{proposition}\label{invsolpres}
 Let $f\in\mathcal{A}(U)$ and $\emptyset\not=U^*\subseteq U$ open, moreover let  $\Phi:\,U^* \to V\subseteq \mathbb K^n$ be solution-preserving from $\dot x=f(x)$ into an equation $\dot x=g(x)$. Then  for every nonnegative integer $s$ the set
\[
 W_s:=\{z\in U^*:\;\emph{\text{rank}}\;D\Phi(z)=s\}
\]
is invariant for $\dot x=f(x)$. 
\end{proposition}
\begin{proof}
The hypothesis implies the identity
\[
 \Phi(F(t,y)) = G(t,\Phi(y))
\]
for all $y\in U^*$ and all $t$ near $0$.
Differentiate with respect to $y$ to obtain
\[
D\Phi(F(t,y)) D_2F(t,y) = D_2G(t,\Phi(y)) D\Phi(y)\,.
\]
Now $D_2F(t,y)$, being the solution of the linear variational equation 
\[
\dot{X}=Df(F(t,y))X,\quad X(0)=I,
\]
 is always invertible and the same property holds for $D_2G(t,\Phi(y))$. Therefore
\[
 {\rm rank}\,D\Phi(F(t,y))={\rm rank}\,D\Phi(y) \text{  for all  }t.
\]
\end{proof}

\begin{example}\label{invsolex}{\em 
Let $h(x):=\begin{pmatrix}
             2x_1 \\ -2x_2 \\ 3x_3 \\ -3x_4
              \end{pmatrix}$ in $\mathbb K^4$ and $f$ analytic in $0$ with $[h,f]=0$. 
One verifies that $\phi_1(x):=x_1x_2$, $\phi_2(x):=x_3x_4$, $\phi_3(x):= x_1^3x_4^2$ and $\phi_4(x):=x_2^3x_3^2$ satisfy $X_h(\phi_i)=0$, $1\leq i\leq 4$, Moreover (as will be seen in Chapter 4) any analytic first integral of $h$ in a neighborhood of $0$ can be written as $\sigma(\phi_1,\ldots,\phi_4)$ with some analytic function $\sigma$. By Proposition \ref{redbyinv} one finds that 
\[
\Phi(x):= \begin{pmatrix}
           		 x_1x_2 \\ x_3x_4 \\ x_1^3x_4^2 \\ x_2^3x_3^2
          				 \end{pmatrix}
\]
is solution-preserving from $\dot x=f(x)$ to some analytic differential equation $\dot x=g(x)$. We apply Proposition \ref{invsolpres} with
\[
 D\Phi(x)=\begin{pmatrix}
           x_2 & x_1 & 0 & 0  \\
	   0 & 0 & x_4 & x_3  \\
	   3x_1^2x_4^2 & 0 & 0 & 2x_1^3x_4  \\
	   0 & 3x_2^2x_3^2 & 2x_2^3x_3 & 0
          \end{pmatrix}.
\]
Computing minors, one finds that always ${\rm rank}\,D\Phi(x)\leq 3$, and furthermore
\begin{align*}
\text{rank}\;D\Phi(x) & \leq 2\Longleftrightarrow\; x_1=x_3=0\;\;\text{or}\;\;x_2=x_4=0\;\;\text{or}\;\;x_1=x_2=0\;\;\text{or}\;\;x_3=x_4=0\,;  \\
\text{rank}\;D\Phi(x) & \leq 1\Longleftrightarrow\; x_1=x_2=0\;\;\text{or}\;\;x_3=x_4=0\,;\\
\text{rank}\;D\Phi(x) &=0\Longleftrightarrow x=0.
\end{align*}
This holds regardless of the special form of the symmetric differential equation $\dot x=f(x)$; any symmetric vector field  must admit these ``invariant sets forced by symmetries''.}
\end{example}
\section{Further remarks and notes}
\begin{itemize}
\item Section 3.1 is loosely based on \cite{WMul}; the classical results are due to Lie \cite{Li 1} and Bianchi \cite{Bi}. See also the standard references mentioned earlier. Theorem \ref{frobred}(a) is stated (in the given form) in Hermann \cite{Her}. Concerning invariant sets, one may consult Field \cite{Fi} as well as Abud and Sartori \cite{AS} for systems admitting compact symmetry groups (keyword stratification); a discussion for linear algebraic symmetry groups is contained in \cite{GSW}. Integrability, while related to symmetry properties, is a vast field in its own right; see e.g.\ the monograph \cite{Zh17} by Zhang.
\item As for an extension to the smooth case, practically everything in Section 3.1 should be handled with care (locally, near nonstationary, resp. maximal rank points, things are ok), while practically everything in Section 3.2 carries over.
\end{itemize}
\chapter{Linear symmetries and normal forms}
In this chapter we generalize Lie's ``reverse engineering'' by starting from a linear group $G\subseteq GL(n,\mathbb K)$ and discussing analytic differential equations that admit $G$ as a symmetry group. We remark that there are deeper reasons for considering linear group actions: For instance every compact group action on a manifold is -- loosely speaking -- equivalent to a linear group action on a submanifold of some $\mathbb R^n$. We discuss toral groups in some detail and then turn to differential equations in Poincar\'e-Dulac normal form, a class of equations which admit toral symmetry groups. Beyond symmetry, differential equations in normal form possess further properties which we discuss in the last part of the chapter.\\
In contrast to the previous chapters, the presentation in Chapter 4  is no longer essentially self-contained, and some of the results are not standard. Some notions are introduced in a rather cursory manner, and some proofs are omitted, in part because the technical build-up would be too expansive.
\section{Linear symmetry groups}
In the present section we consider analytic ordinary differential equations \eqref{ode} which admit every element of a given a subgroup $G\subseteq GL(n,\mathbb K)$ as a symmetry. We will focus on linear algebraic groups; these are subgroups of $GL(n,\mathbb K)$ which are also algebraic subvarieties of $\mathbb K^{(n,n)}$. (Most of the familiar matrix groups fall in this class.) We start with some elementary facts.
\begin{lemma} Let the analytic ordinary differential equation  $\dot x=f(x)$ be given on $U\subseteq \mathbb K^n$. Then the following hold.
\begin{enumerate}[(a)]
\item A given $T\in GL(n,\mathbb K)$ defines a symmetry of \eqref{ode} via $x\mapsto Tx$  if and only if $f(Tx)=Tf(x)$ on $U$.
\item The set
\[
\widehat G:=\left\{T\in GL(n,\mathbb K);\,f(Tx)=Tf(x)\text{  on  }U\right\}
\]
is an algebraic subgroup of $GL(n,\mathbb K)$.
\end{enumerate}
\end{lemma}
\begin{proof}
Part (a) follows from Lemma \ref{solprescrit}. For part (b) the subgroup property is straightforward. To see that this subgroup is algebraic, assume $0\in U$ with no loss of generality and consider the Taylor expansion $f=\sum_{j\geq 0}f_j$, with each $f_j$ a homogeneous polynomial map of degree $j$. Then the defining condition is equivalent  to $Tf_j(x)=f_j(Tx)$ for all $x\in \mathbb K^n$, $0\leq j<\infty$. Writing out these conditions for the matrix entries of $T$ and the (given) coefficients of the $f_j$, one sees that the entries of $T$ are characterized by polynomial equations.
\end{proof}
So, from now on it suffices to restrict attention to linear algebraic symmetry groups of \eqref{ode}. To a linear algebraic group $G$ we associate  its {\em Lie algebra} 
\begin{equation}\label{liealgeq}
\mathcal{L}=\mathcal{L}(G):=\left\{B\in\mathbb K^{(n,n)};\,\exp(sB)\in G\text{  for all  } s\right\}.
\end{equation}
We adjust resp. introduce some notation.
\begin{definition} Let $G$ be a linear algebraic group. Then we denote by 
 $I(G)$ the set of all polynomial invariants of $G$, i.e. all polynomials $\varphi:\mathbb K^n\to \mathbb K$ such that $\varphi\circ T^{-1}=\varphi$ for all $T\in G$. Moreover we denote by $\mathcal{P}(G)$ the set of all $G$-symmetric polynomial vector fields on $\mathbb K^n$.  
\end{definition}
Next we note some structural properties.
\begin{lemma} Let $G$ be a linear algebraic group. Then  
 $I(G)$ is a subalgebra of $\mathbb K\left[x_1,\ldots, x_n\right]$, and  $\mathcal{P}(G)$ is a Lie subalgebra of the Lie algebra of all polynomial vector fields. \\
Moreover, if $G$ is a symmetry group for a polynomial ordinary differential equation $\dot x=f(x)$ then for any $\varphi\in I(G)$ one has $X_f(\varphi)\in I(G)$.
\end{lemma}
\begin{proof}
The first two statements are straightforward. For the last,  differentiate $\varphi(Tx)=\varphi(x)$ to obtain $D\varphi(Tx)T = D\varphi(x)$, hence
\[
\begin{array}{rcl}
 X_f(\varphi)(Tx)&=&D\varphi(Tx)f(Tx)=D\varphi(Tx) T \,T^{-1}f(Tx)\\
&=&D\varphi(x)f(x)= X_f(\varphi)(x)
\end{array}
\]
\end{proof}
An immediate consequence is the following variant of Proposition \ref{redbyinv}.
\begin{proposition}\label{polyredbyinv}
Let $G$ be a linear algebraic group.
If the $\mathbb K$-Algebra $I(G)$ admits a finite system $\varphi_1,\ldots,\varphi_r$ of generators and $f\in \mathcal{P}(G)$, then 
\[
\Phi:=\begin{pmatrix}
						\varphi_1 \\ \vdots \\ \varphi_r 
                                                      \end{pmatrix}
\]
is solution-preserving from $\dot x=f(x)$ to some polynomial differential equation $\dot x=g(x)$ in $\mathbb K^r$. The Zariski closure $Y$ of $\Phi(\mathbb K^n)$ is invariant for $\dot x=g(x)$.
\end{proposition}
\begin{proof}
As noted above, $X_f(\varphi_j)\in I(G)$ for all $j$; hence there exist $\gamma_j\in\mathbb K\left[x_1,\ldots,x_r\right]$ such that $X_f(\varphi_j)=\gamma_j(\varphi_1,\ldots,\varphi_r)$, $1\leq j \leq r$. Setting
\[
g:=\begin{pmatrix}
    \gamma_1 \\ \vdots \\ \gamma_r                                      \end{pmatrix},
\] 
the condition $D\Phi(x)f(x)=g(\Phi(x))$ is therefore satisfied.   \\
It suffices to prove the second assertion for the case $\mathbb K=\mathbb C$. The image of $\Phi$ contains a Zariski-open subset $\widetilde{Y}$ of $Y$, which is also open and dense with respect to the restriction of the norm topology. Clearly $\widetilde Y$ is invariant for $\dot x=g(x)$, and this also holds for the closure $Y$.
\end{proof}

\begin{example}\label{mpresex} {\em Let $m$ and $p$ be relatively prime positive integers, and
\[
 T_{\lambda}x= \begin{pmatrix}
                       \lambda^p x_1 \\ \lambda^{-p} x_2  \\
			\lambda^m x_3 \\ \lambda^{-m}x_4
                         \end{pmatrix},\quad \lambda\in\mathbb K^*, \text{  and  }G=\{T_{\lambda}:\; \lambda\in\mathbb K^*\}.
\]
The Lie algebra $\mathcal{L}(G)$ is spanned by $B$, with
\[
Bx=\begin{pmatrix}
						px_1 \\ -px_2 \\ mx_3 \\ -mx_4
                                                         \end{pmatrix}.
\]
For $\varphi(x)= x_1^{m_1}\cdot x_2^{m_2}\cdot x_3^{m_3}\cdot x_4^{m_4}$ one computes $\varphi(T_{\lambda}x)=\lambda^{p(m_1-m_2)+m(m_3-m_4)}\varphi(x)$. As a consequence, the elements of $I(G)$ are precisely the $\mathbb K$-linear combinations of monomials
\[
 x_1^{d_1}\cdots x_4^{d_4}\;\; \text{such that}\;\,p(d_1-d_2)+m(d_3-d_4)=0.
\]
From this one sees that
\[
 \varphi_1(x)=x_1x_2,\quad \varphi_2(x) = x_3x_4,\quad \varphi_3(x) = x_1^mx_4^p\;\;\text{und}\;\;\varphi_4(x)= x_2^mx_3^p
\]
generate the algebra $I(G)$ (and these form a smallest set of generators).
Hence, if $\dot x=f(x)$ is $G$-symmetric then $\Phi=\begin{pmatrix}
                                                 \varphi_1 \\ \vdots \\ \varphi_4
                                                                \end{pmatrix}$ 
maps solutions of a polynomial differential equation $\dot x=f(x)$ to solutions of some polynomial system $\dot x = g(x)$ in $\mathbb K^4$. (Compare also Example \ref{invsolex} for the case $p=2,\, m=3$.) \\
The image of $\Phi$ is contained in the algebraic variety $Y=\{y\in\mathbb K^4:\;y_1^my_2^p-y_3y_4=0\}$, as follows from the relation $\varphi_1^m\varphi_2^p-\varphi_3\varphi_4=0$. Since $\dim Y=3$ one obtains (as intuitively expected) a reduction of dimension by one for symmetric systems.}
\end{example}
The statement of Proposition \ref{polyredbyinv} raises a few questions, which we will address next.
\begin{remark}\label{fingenrem}{\em  Proposition \ref{polyredbyinv} requires that the invariant algebra of a given linear algebraic group is finitely generated. The following facts are known.
\begin{enumerate}[(a)]
\item There exist linear algebraic groups whose invariant algebra is not finitely generated. The first example goes back to Nagata.
\item On the other hand, any algebraic subgroup of the orthogonal group $O(n,\mathbb R)$ (thus, a compact linear algebraic group) always admits a finite generator system of the invariant algebra. Moreover, $\mathcal{P}(G)$ is a finitely generated module over $I(G)$. This result goes back to E.~Noether.
\item  More generally, one may consider reductive linear algebraic groups. We introduce this notion following {Schwarz} \cite{Sc 2}, which is the most appropriate version for our purposes.
\begin{itemize}
\item Let $H\subseteq GL(n,\mathbb R)$ be a linear algebraic group, with defining equations $\tau_1,\ldots,\tau_m$. Then the vanishing set $H_{\mathbb C}\subseteq GL(n,\mathbb C)$ of $\tau_1,\ldots,\tau_m$ (which are now considedered as elements of $\mathbb C[x_1,\ldots,x_n]$) is a complex algebraic group, and  $H$ is Zariski dense in $H_{\mathbb C}$.   
\item We call a complex algebraic group $G$ \emph{reductive} if there exists a compact subgroup $K$ of $GL(n,\mathbb R)$ such that $G\simeq K_{\mathbb C}$. A real algebraic group $H$ is called reductive if $H_{\mathbb C}$ is reductive. This ``un-algebraic'' notion of reductivity differs from the usual one (see e.g. Springer \cite{Sp}), but the notions can be shown to be equivalent. 
\item For example $\{0\}$ is the only compact subgroup of $(\mathbb C,+)$, hence $(\mathbb C,+)$ is not reductive. On the other hand, $(\mathbb C^*,\,\cdot\,)$ is reductive (take the compact subgroup $S^1$).
\end{itemize}
Using this definition and the finite generation property for compact groups it is now easy to prove that the invariant algebra of every reductive linear algebraic group is  finitely generated. Moreover, $\mathcal{P}(G)$ is a finitely generated module over $I(G)$.

\end{enumerate}
}
\end{remark}
Finite generation guarantees that Proposition \ref{polyredbyinv} is applicable, but Example \ref{mpresex} already shows that the number of generators may be larger than the dimension of the reduced system (restricted to an algebraic variety). We illustrate next that this dimension may be substantially larger.
\begin{example}{\em 
\begin{enumerate}[(a)]
\item Let $m\in\mathbb N$, $I_m$ the $m\times m$ identity matrix and
\[
G=\left\{\begin{pmatrix} a\cdot I_m& 0\\
                                     0&a^{-1}\cdot I_m\end{pmatrix}; a\in\mathbb K^*\right\}\subseteq GL(2m,\,\mathbb K).
\]
Denoting the diagonal elements by $x_1,\ldots,x_{2m}$, one can verify that the invariant algebra admits the (smallest) generator set
\[
\left\{ \gamma_{ij}:=x_ix_{m+j}; \,1\leq i,\,j\leq m\right\}
\]with $m^2$ elements. There exist many relations between these generators, viz.
\[
 \gamma_{ij}\cdot \gamma_{k\ell}= \gamma_{i\ell}\cdot  \gamma_{kj},
\]
and from these one may find that the image of the reduction map $\Phi$ is an algebraic variety of dimension $2m-1$. (See the following section for more details.)
But one cannot avoid the problem of dealing with a high dimensional embedding space, and with a rather complicated image of $\Phi$.
\item The irreducible $11$-dimensional representation of $SL(2,\mathbb C)$ admits a smallest generator system with 106 elements; see Brouwer and Popoviciu \cite{BrPo}. Thus the embedding space of the variety $Y$ from Proposition \ref{polyredbyinv} (which has dimension $8$) is 106-dimensional. No relations between the generators are given in \cite{BrPo}, hence the image of the reduction map would not be readily available.
\end{enumerate}
}
\end{example}
Examples like these explain why reduction by invariants is sometimes unfeasible, and not very popular with some practitioners. An algebraic approach to  escape this dilemma is proposed in \cite{SchWa}.\\
Finally, there is the question whether one can extend  Proposition \ref{polyredbyinv} from polynomial to analytic differential equations. While the extension from polynomials to formal power series is straightforward, convergence issues are a highly nontrivial problem. The following result is due to Luna \cite{Lu}, building on work by Schwarz \cite{Sc 1} and Poenaru \cite{Po} for smooth real functions and vector fields with compact symmetry groups.
\begin{theorem}\label{anaredbyinv}
Let $G\subseteq GL(n,\mathbb K)$ be reductive, and $\varphi_1,\ldots,\varphi_r$ a generator system for $I(G)$.
\begin{enumerate}[(a)]
 \item For every $G$-invariant analytic function germ $\psi$ in $0$  (with values in $\mathbb K$) there exists a power series $\sigma$ in $r$ variables, with nonempty domain of convergence, such that $\psi=\sigma \circ \Phi$. 
\item Furthermore, if $p_1,\ldots,p_s$ is a generator system of the $I(G)$-module $\mathcal{P}(G)$, then for every germ of a $G$-invariant analytic vector field $f$ in $0$ there exist $G$-invariant analytic function germs $\psi_1,\ldots,\psi_s$ such that $f=\sum\limits^s_{j=1}\psi_jp_j$.
\end{enumerate}
\end{theorem}
\section{Toral groups}
In this section we will discuss one rather special, but important class of linear algebraic groups, the toral groups. We first need some preparations concerning the action of linear vector fields on polynomials and polynomial vector fields.

 Let $B\in\mathbb K^{(n,n)}$ and consider the linear map $\mathbb K^n\to\mathbb K^n,\,x\mapsto Bx$. (In the following we will use terminology rather loosely and identify the map with the matrix; for instance we will write $X_B$.) Moreover let
\begin{equation}\label{Jordan}
B=B_s+B_n,\quad B_s\text{  semisimple},\quad B_n\text{  nilpotent and  } \left[B_s,\,B_n\right]=0
\end{equation} 
be the Jordan-Chevalley decomposition of $B$. (If $B$ is in Jordan canonical form then $B_s$ is just the diagonal part of $B$.)
We note some facts about the operations $X_B$ and ${\rm ad}\,B$.
\begin{lemma}\label{operationsb} Let $B$ be as above.
\begin{enumerate}[(a)]
\item $X_B$ maps each space $S_m$ of homogeneous polynomials of degree $m$ into itself, and ${\rm ad}\,B$ maps each space $\mathcal{P}_m$ of homogeneous polynomial vector fields of degree $m$ into itself.
\item Let  $e_1,\ldots, e_n$ be an eigenbasis of $B_s$ (possibly after complexification) and denote by $x_1,\ldots,x_n$ the corresponding coordinates. Then with $B_s e_i = \lambda_i e_i$, $1\leq i\leq n$, the following hold.
\begin{itemize}
\item For $\varphi(x):=x_1^{m_1}\cdots x_n^{m_n}$ one has $X_{B_s}(\varphi)=(m_1\lambda_1+\cdots +m_n\lambda_n)\varphi$;
\item For $ p(x):=x_1^{d_1}\cdots x_n^{d_n}e_j$ one has $[B_s,p]=(d_1\lambda_1+\cdots +d_n\lambda_n-\lambda_j)p$.
\end{itemize}
\item 
\begin{itemize}
 \item $X_{B_s}$ acts as a semisimple linear map on each $S_m$, $X_{B_n}$ acts s a nilpotent linear map on $S_m$, and  $X_B=X_{B_s}+X_{B_n}$ is the Jordan-Chevalley decomposition for $X_B$ on $S_m$. 
 \item ${\rm ad}\,B_s$ acts as a semisimple linear map on each $\mathcal{P}_m$,  ${\rm ad}\,B_n$ acts s a nilpotent linear map on $\mathcal{P}_m$ and ${\rm ad}\,B={\rm ad}\,B_s+{\rm ad}\,B_n$  is the Jordan-Chevalley decomposition for ${\rm ad}\,B$ on $\mathcal{P}_m$.
\end{itemize}
\end{enumerate}
\end{lemma} 
\begin{proof}
Parts (a) and (b) follow directly from differentiation rules and simple computations.  As for part (c), the semisimplicity of $X_{B_s}$ follows from (b), where a basis of ``eigenvectors'' is given, and $X_{B_s},\,X_{B_n}$ commute by Proposition \ref{liebproperties}. To verify that $X_{B_n}$ acts nilpotently, one first shows by induction: If $\phi\in S_m$ and $\widehat\phi$ is the symmetric $m$-linear map from $\mathbb K^m$ to $\mathbb K$ such that $\widehat\phi(x,\ldots,x)=\phi(x)$ for all $x$, then $X_{B_n}^k(\phi)(x)$ is a $\mathbb K$-linear combination of terms
\[
\widehat \phi(B_n^{i_1}x,\ldots,B_n^{i_r}x) \text{  with all  }i_\ell\geq 0,\,\sum i_\ell=k.
\]
This implies $X_{B_n}^{mq}(\phi)=0$ if $B_n^q=0$. The proof for ${\rm ad}\, B$ runs similarly.
\end{proof}
For further reference we specialize, resp. introduce, some definitions.
\begin{definition}\label{Bdefs} Let $B\in\mathbb K^{(n,n)}$.
\begin{enumerate}[(i)]
\item We call 
\[
I(B)=I_0(B):=\left\{\phi\in\mathbb K[x_1,\ldots,x_n];\,X_B(\phi)=0\right\}
\]
the set of  (polynomial) {\em invariants} of $B$.
\item More generally, for every $\chi\in\mathbb K$ we define
\[
I_\chi(B):=\left\{\phi\in\mathbb K[x_1,\ldots,x_n];\,X_B(\phi)=\chi\phi\right\}
\]
and call this the set of (polynomial) $\chi$-{\em semi-invariants} of $B$.
\item Finally, we call
\[
\mathcal{C}(B):=\left\{p\in\mathcal{P};\, \left[B,\,p\right]=0\right\}
\]
the (polynomial) {\em centralizer} of $B$.
\end{enumerate}
\end{definition}
With the product rule one easily verifies:
\begin{lemma}\label{modmultlem}
Let $B\in\mathbb K^{(n.n)}$. Then for all $\chi,\,\eta\in\mathbb K$ one has 
\[
I\chi(B)\cdot I_\eta(B)\subseteq I_{\chi+\eta}(B).
\]
In particular $I(B)$ is a commutative and associative algebra and $I_\chi(B)$ is a module over $I(B)$.
\end{lemma}
We now define the central object of the present section, in a somewhat unusual manner.
\begin{definition}\label{deftoral}
We call a connected linear algebraic group $G\subseteq GL(n,\mathbb K)$ {\em toral} if there exists a semisimple $B\in\mathbb K^{(n,n)}$ such that $\left\{\exp(tB);\,t\in\mathbb R\right\}$ is Zariski-dense in $G$.
\end{definition}
Up to conjugacy, toral groups have a rather simple structure.
\begin{proposition}\label{extoral}
 Let $1\leq s<n$ and $M=\left(m_{ij}\right)\in\mathbb Z^{(s,n)}$ with ${\rm rank}\,M=s$. Then
\[
G:=\left\{{\rm diag}\left(a_1^{m_{11}},\ldots,a_1^{m_{1n}}\right)\cdots {\rm diag}\left(a_s^{m_{s1}},\ldots,a_s^{m_{sn}}\right); \, a_i\in\mathbb C^*\right\}
\]
is a toral subgroup of $GL(n,\mathbb C)$. Defining relations for this group are given by $t_1^{d_1}\cdots t_n^{d_n}=1$ for all integer vectors $(d_1,\ldots,d_n)^{\rm tr}\in \ker M$,
 where the $t_i$ denote the matrix diagonal entries, and all off-diagonal entries zero.
\end{proposition}
\begin{proof}{
\begin{enumerate}
\item To verify that $G$ is a toral group in the sense of Definition \ref{deftoral}, let $\beta_1,\ldots, \beta_s\in \mathbb R$ be linearly independent over the rationals $\mathbb Q$, and define
\[
B:=\beta_1{\rm diag}\left(m_{11},\ldots,m_{1n}\right)+\cdots+\beta_s{\rm diag}\left(m_{s1},\ldots,m_{sn}\right).
\]
Then with $\lambda_k:=\beta_1m_{1k}+\cdots +\beta_sm_{sk}$, the diagonal entries of $\exp(tB)$ are
\[
\exp\left(t\lambda_k\right) =\exp(t\beta_1)^{m_{1k}}\cdots\exp(t\beta_s)^{m_{sk}},
\]
which shows directly that $\exp(tB)\in G$.
\item We show next that the set of all $\exp(tB)$ is Zariski dense in $G$, applying Lemma \ref{operationsb} to the diagonal entries of the matrices. Thus let $\psi\in\mathbb C[x_1,\ldots,x_n]$ such that $\psi(\exp(tB))=0$ for all $t$. We will show that $\psi$ vanishes on all  ${\rm diag}\,(t_1, \ldots, t_n)\in G$. By Lemma \ref{operationsb} one may write $\psi=\sum_\chi\psi_\chi$, with $X_B(\psi_\chi)=\chi\cdot\psi_\chi$, and $\psi$ is a linear combination of monomials $t_1^{m_1}\cdots t_n^{m_n}$ with $\chi=\sum m_i\lambda_i$. Differentiating the relation $\psi(\exp(tB))=0$, one finds
\[
0=X_B^k(\psi)(\exp(tB))=\sum \chi^k\psi_\chi(\exp(tB)),\quad\text{all  } k\geq 0,
\]
and by a Vandermonde argument this implies 
\[
\psi_\chi(\exp(tB))=0 \text{  for all  }t.
\]
Now, if two monomials $t_1^{m_1}\cdots t_n^{m_n}$ and $t_1^{\ell_1}\cdots t_n^{\ell_n}$ appear in some $\psi_\chi$ with nonzero coefficients, then 
\[
(m_1-\ell_1)\lambda_1+\cdots+(m_n-\ell_n)\lambda_n=0
\]
from above, and thus $(m_1-\ell_1,\ldots,m_n-\ell_n)\in \ker M$. Therefore
\[
t_1^{m_1-\ell_1}\cdots t_n^{m_n-\ell_n}=1\text{  for all elements of }G.
\]
 This shows the assertion.
\end{enumerate}
}
\end{proof}
\begin{remark}{\em In the course of the proof we also have seen that $\phi\in I(G)$ if and only if $X_B(\phi)=0$.}
\end{remark}
By the first characterization in Proposition \ref{extoral}, $G$ contains compact subgroups $\{{\rm diag}\left(b_k^{m_{k1}},\ldots,b_k^{m_{kn}});\, |b_k|=1 \right)\}$, and we obtain:
\begin{lemma}\label{torredlem}
Every toral group is reductive.
\end{lemma}
To show finite generation properties for toral groups, one could appeal to this lemma and to Remark \ref{fingenrem}, but we will give direct (partly constructive) proofs. These proofs imitate arguments used e.g. for compact groups, but they are technically less involved.
\begin{proposition}\label{torfinprop} Let $G$ be toral and $B\in \mathbb K^{(n,n)}$ as in Definition \ref{deftoral}, in particular $B$ is semisimple. Then the following hold.
\begin{enumerate}[(a)]
\item $I(B)$ is a finitely generated algebra, and $I(B)=I(G)$.
\item For every $\chi\in\mathbb K$ the $I(B)$-module $I_\chi(B)$ is finitely generated.
\item The $I(B)$-module $\mathcal{C}(B)$ is finitely generated.
\end{enumerate}

\end{proposition}
\begin{proof}
We first prove part (b). Let $\{\psi_i;\,i\in I\}$ be an arbitrary generator set of $I_\chi(B)$, and every $\psi_i$ homogeneous without loss of generality. By Hilbert's {\em Basissatz} the ideal generated by these $\psi_i$  admits a finite set of generators which we rename $\{\psi_1,\ldots,\psi_r\}$. Now let $\rho\in I_\chi(B)$; we may assume that $\rho$ is homogeneous of degree $d$ with no loss of generality. Then there exist $\mu_i=\sum_\eta\mu_{i,\eta}\in\mathbb K[x_1,\ldots,x_n]$ (with $\mu_{i.\eta}\in I_\eta(B)$  homogeneous with no loss of generality) such that 
\[
\rho=\sum_i\mu_i\psi_i=\sum_i\mu_{i,0}\psi_i +\sum_{\eta\not=0}\sum_i\mu_{i,\eta}\psi_i.
\]
Now $\rho$ and the first term on the right hand side are elements of $I_\chi(B)$, while the remaining terms lie in the sum of subspaces $I_{\chi+\eta}(B)$ by Lemma \ref{modmultlem}. Since $S_d$ is the direct sum of subspaces $I_\beta(B)$, one has necessarily that
\[
\rho=\sum_i\mu_{i,0}\psi_i,
\]
thus the $\psi_i$ generate $I_\chi(B)$ as an $I(B)$-module. Part (c) is proven by a variant of this argument, using a variant of Lemma \ref{modmultlem}.\\ To prove part (a), continue the proof of (b) with $\chi=0$, using induction on the degree $m$ of $\rho$ to show that each $\mu_{i,0}$ lies in $\mathbb K[\psi_1,\ldots,\psi_r]$.
\end{proof}
As a consequence, Proposition \ref{polyredbyinv} and Theorem \ref{anaredbyinv} are applicable to toral groups.
 We also note:
\begin{corollary}\label{tortrivinv}
Whenever $I(B)=\mathbb K$ then $\mathcal{C}(B)$ is a finite dimensional vector space over $\mathbb K$.
\end{corollary}
There exist toral groups with trivial invariant algebra, such as the following:
\begin{example}{\em 
Let $d_1,\ldots,d_n$ be positive integers and 
\[
G:=\left\{{\rm diag}\,(a^{d_1},\ldots,a^{d_n});\,a\in \mathbb C^*\right\}, \quad H:=\left\{{\rm diag}\,(a^{d_1},\ldots,a^{d_n});\,a\in \mathbb R_{>0}^*\right\}.
\]
Then $G$ is the complexification of $H$ and $I(G)$ as well as $I(H)$ is trivial. To verify this, consider $B={\rm diag}\,(d_1,\ldots,d_n)$, and note $\sum m_id_i>0$ for all tuples of nonnegative integers $m_1,\ldots, m_n$ with positive sum.}
\end{example}
For toral groups with trivial invariant algebra, reduction by invariants is not applicable. But one may refine Corollary \ref{tortrivinv} to see that differential equations admitting symmetry groups of this type admit elementary solutions.
\begin{proposition}\label{torelemprop}
Let $G$ be a toral group with trivial invariant algebra, and let $f\in\mathbb{C}(G)$. Then there exist $r\geq 1$ and subspaces $W_1,\ldots,W_r$ with $\mathbb K^n=W_1\oplus\cdots\oplus W_r$, such that for $x=\begin{pmatrix}x^{(1)}\\ \vdots\\x^{(r)}\end{pmatrix}$ corresponding to this decomposition, $\dot x=f(x)$ takes the form
\[
\begin{array}{rcccl}
\dot x^{(1)}&=& A_1x^{(1)}& & \\
\dot x^{(2)}&=&A_2x^{(2)}&+& q_2(x^{(1)})\\
             &\vdots& & & \\
\dot x^{(r)}&=&A_rx^{(r)}&+&q_r(x^{(1)},\ldots,x^{(r-1)})
\end{array}
\]
with matrices $A_i$ (of appropriate size) and polynomial maps $q_i$. In particular, every solution of the differential equation is elementary.
\end{proposition}
\begin{proof}Let $B$ be as in Definition \ref{deftoral}, with eigenvalues $\lambda_1,\ldots,\lambda_n$. We may assume that $B$ is diagonal. Then $\mathcal{C}(B)$ is spanned by monomial vector fields
\[
x_1^{m_1}\cdots x_n^{m_n}e_j\quad\text{  with  } \sum_im_i\lambda_i=\lambda_j.
\]
For a given monomial, call $\sum m_i$ the {\em length} of the corresponding relation. Let $W$ be the subspace spanned by all $e_j$ admitting a relation of maximal length, say $W=\left<e_{s+1},\ldots,e_n\right>$ with no loss of generality. Then $x_{s+1},\ldots, x_n$ do not appear in any vector monomial $x_1^{m_1}\cdots x_n^{m_n}e_j\in \mathcal{C}(B)$ with $j\leq s$, and they appear only in relations of degree one (thus, linearly) when $j>s$. This holds true because in a relation
\[
\sum \ell_k\lambda_k=\lambda_p
\]
with $\ell_j\not=0$ one may substitute $ \sum_im_i\lambda_i=\lambda_j$ to obtain a relation of greater length than either when both have length $>1$.\\
Thus one has, with $y=\begin{pmatrix}x_1\\ \vdots\\x_s\end{pmatrix}$ and  $z=\begin{pmatrix}x_{s+1}\\ \vdots\\x_n\end{pmatrix}$, a decomposition
\[
\begin{array}{rcccl}
\dot y&=& h_1(y) & & \\
\dot z&=& Az&+& h_2(y)
\end{array}
\]
with a matrix $A$ and polynomials $h_i$. Proceed by induction. To verify the statement about elementary solutions, solve the system from the top down.
\end{proof}
\begin{remark}\label{splitrem}{\em 
For any toral $G$ there exists a maximal $B$-invariant subspace $V$ such that every element of $I(B)$ is constant on $V$, and with its $B$-invariant complementary subspace $U$ one obtains a decomposition $\mathbb K^n=U\oplus W$. If  both subspaces are nontrivial, there is a corresponding decomposition of any $G$-symmetric differential equation in the form
\[
\begin{array}{rcl}
\dot u&=& g(u)\\
\dot v&=&h(u,v).
\end{array}
\]
Here one has a system (of smaller dimension) for $u$ alone, and the equation for $v$ may be further split up in a manner similar to  Proposition \ref{torelemprop}, but with the coefficients of the $A_i$ and $q_i$ now depending (polynomially or analytically) on $u$. See \cite{Wa 3} for more details.}
\end{remark}
We look more closely at the case $V=\{0\}$.
\begin{proposition}\label{reducdimprop}
Let $M$, $G$ and $B$ be as in Proposition \ref{extoral}, and assume that $V=\{0\}$ is the only $B$-invariant subspace on which every element of $I(B)$ is constant. Then:
\begin{enumerate}[(a)]
\item There exists a monomial $x_1^{d_1}\cdots x_n^{d_n}\in I(B)$ with all $d_i>0$.
\item Given a set of generators $\{\phi_1,\ldots,\phi_r\}$ for $I(B)$,  the Jacobian of
\[
\Phi=\begin{pmatrix}\phi_1\\ \vdots\\ \phi_r\end{pmatrix}:\,\mathbb C^n\to\mathbb C^r
\]
has generically rank $n-s$ (which is equal to the rank of the matrix $M$), and the image of $\Phi$ contains a Zariski-open and dense subset of an $n-s$-dimensional algebraic variety $Y\subseteq \mathbb C^r$.
\end{enumerate}
\end{proposition}
\begin{proof} By assumption, for every index $j$ there exists an element of $I(B)$ which is a multiple of $x_j$. Taking a suitable product shows part (a). For the proof of part (b), consider a basis of the kernel of $M$ with integer entries and take these basis elements as rows of a matrix $P\in\mathbb K^{(n-s,n)}$. With the entries of the matrix $P$ and a positive integer $\ell$, define
\[
\phi_k^*:=x_1^{p_{k1}+\ell d_1}\cdots x_n^{p_{kn}+\ell d_n}
\]
and 
\[
\Phi^*=\begin{pmatrix} \phi_1^*\\ \vdots \\  \phi_{n-s}^*\end{pmatrix},\text{  with } D\Phi^*\left(\begin{pmatrix} 1\\ \vdots\\ 1\end{pmatrix}\right)=\begin{pmatrix}p_{11}+\ell d_1 & \cdots & p_{1n}+\ell d_n\\
\vdots & & \vdots \\
p_{n-s,1}+\ell d_1 & \cdots & p_{n-s,n}+\ell d_n\\
\end{pmatrix}.
\]
For some $\ell>0$ the latter matrix has full rank. Since every $\phi_j^*$ is a polynomial in the $\phi_i$, the generic rank of $D\Phi(x)$ must be $\geq s$. On the other hand, every $\phi_i$ is a rational function of the $x_1^{p_{k1}}\cdots x_n^{p_{kn}}$, $1\leq k\leq n-s$, hence the rank of $D\Phi(x)$ is $\leq s$. The remaining assertion follows from familiar properties of morphisms of algebraic varieties.
\end{proof}
We close this section with a few small examples.
\begin{example}\label{toralsymex}{\em 
Let $\mathbb K=\mathbb C$ (for the sake of simplicity), $B={\rm diag}\,(\lambda_1,\ldots,\lambda_n)$, and $G$ the toral group corresponding to $B$ as in Proposition \ref{extoral} above. The centralizer of $B$ can be determined according to Lemma \ref{operationsb}.
\begin{itemize}
 \item If the equation $m_1\lambda_1+\cdots +m_n\lambda_n-\lambda_j=0$ (with $\sum m_i=r$, $1\leq j\leq n$) has no solution in nonnegative integers for $r\geq 2$ (for instance, $\lambda_1,\ldots,\lambda_n$ are linearly independent over $\mathbb Q$), then the only differential equations admitting $G$ are linear equations $\dot x=Cx$. If the $\lambda_i$ are pairwise different then $C$ is necessarily diagonal.
 \item A simple example to illustrate Proposition \ref{torelemprop} is given by $Bx= \begin{pmatrix}
                            x_1  \\ 2x_2
                             \end{pmatrix}$: Every $G$-symmetric differential equation has the form
\[
\dot x=Cx+\alpha\cdot\begin{pmatrix}
                       0 \\ x_1^2
                        \end{pmatrix},\quad C\text{  linear}, \alpha\in \mathbb C,
\]
which is easily solved line by line.
\item Let $n\geq 2$ and $\lambda_1=0$, $\lambda_2=\cdots=\lambda_n=1$, thus we are in the situation of Remark \ref{splitrem}. For $2\leq i,j\leq n$ let $\gamma_{ij}$ be analytic, moreover $\psi$ analytic with $\psi(0)=0$. Then every differential equation
\[
\begin{array}{rcl}
  \dot x_1 &=&\psi(x_1)\hspace{2,8cm}   \\
\dot{\begin{pmatrix}
      x_2 \\ \vdots \\ x_n
     \end{pmatrix}} &=& \bigl(\gamma_{ij}(x_1)\bigr)_{2\leq i,j\leq n}
\cdot \begin{pmatrix}
      x_2 \\ \vdots \\ x_n
     \end{pmatrix}
\end{array}
\]
admits the symmetry group $G$, and conversely every $G$-symmetric differential equation with stationary point $0$ has the form above. One reads off directly the reduction to a one-dimensional equation, with a linear non-autonomous equation remaining.
\item Let 
\[
Bx=\begin{pmatrix}
                                  x_1  \\  -x_2
                                    \end{pmatrix},\quad \varphi(x):=x_1x_2.
\]
Then $\dot x=f(x)=B(x)+\cdots$ admits the symmetry group $G$ if and only if there exist diagonal matrices $C_i\in\mathbb C^{(2,2)}$ such that
\[
 f(x)= C_0\,x+\sum_{i\geq 1} \varphi(x)^i\,C_ix.
\]
 Proposition \ref{polyredbyinv} and Theorem \ref{anaredbyinv} show that 
 $\varphi$ is solution-preserving from $\dot x=f(x)$ to a one-dimensional equation $\dot x=g(x)$. Indeed, one finds $X_{C_i}(\phi)={\rm tr}\, C_i\cdot \phi$ and therefore $g(x)=\sum\limits_{i\geq 0}{\rm tr}\, C_i \cdot x^{i+1}$.
\end{itemize}
}
\end{example}
\section{Poincar\'e-Dulac normal forms}
We first recall the normal form problem for an analytic ordinary differential equation \eqref{ode} near a stationary point (which we may assume to be $0$).
Thus we start from a Taylor expansion
\begin{equation}\label{taylorode}
f=A+\sum\limits_{j\geq 2}f_j\,;\quad  A=Df(0),\quad f_j \text{  homogeneous of degree  }j.
\end{equation}
The goal is to find a simpler (in a sense yet to be specified) equation $\dot x=f^*(x)$ and an invertible solution-preserving map $\Phi$ from $\dot x=f^*(x)$ to  $\dot x=f(x)$.  \\
Since $Df^*(0)$ and $A=Df(0)$ are conjugate via $D\Phi(0)$, we may assume that $Df^*(0)=A$, $D\Phi(0)=I$, and we get a Taylor expansion
\begin{equation}\label{taylonofo}
f^*=A+\sum\limits_{j\geq 2}f_j^*.
\end{equation}
One possibility to proceed is degree by degree, as follows: If in the expansion \eqref{taylorode} the terms $f_2,\ldots,f_{r-1}$ are deemed acceptable, then make the ansatz
\[
 f^*=A+f_2+\cdots +f_{r-1}+f_r^*+\cdots, \quad \Phi=I+h_r+\cdots,\, D\Phi(x)=I+Dh_r(x)+\cdots,
\]
and evaluate the condition $D\Phi(x)f^*(x)=f(\Phi(x))$ from Lemma \ref{solprescrit} for every degree: For degrees $1,\ldots,r-1$ equality holds automatically, while at degree $r$ equality holds if and only if
\begin{equation}\label{homoleq}
 f_r^*(x)+Dh_r(x)Ax= Ah_r(x)+f_r(x), \quad \text{i.e.}\quad [A, h_r]=f_r-f_r^*.
\end{equation}
Thus $f_r^*$ must be chosen so that $f_r-f_r^*$ lies in the image of ${\rm ad}\,A$, and for any such choice there exists $h_r$ such that \eqref{homoleq}  is satisfied.
In other words, $f_r^*$ may be chosen from any subspace $W_r\subseteq \mathcal{P}_r$ that satisfies ${\rm im}\,({\rm ad}\,A)+W_r=\mathcal{P}_r$. (Preferably the dimension of $W_r$ should be as small as possible.) From an algebraic perspective, a canonical choice is $W_r=\ker({\rm ad} \,A_s)$. In particular, $\ker({\rm ad} \,A_s)$ is a direct summand for the image whenever $A=A_s$ is semisimple.

We next establish the connection to the usual coordinate-dependent derivation of normal form transformations and normal forms.
\begin{example}\label{nofoevalex}{\em 
Assume that  $A=A_s={\rm diag}\,(\lambda_1,\ldots,\lambda_n)$ is diagonal. We then have
\[
 f_r = \sum \alpha_{m_1,\ldots,m_{n},j}\, x_1^{m_1}\cdots x_n^{m_n} e_j,
\]
summation extending over all $(m_1,\ldots,m_{n},j)$ with $m_1+\cdots +m_n=r$, and $1\leq j\leq n$.   
With the ansatz $h_r=\sum\beta_{m_1,\ldots,m_{n},j}\, x_1^{m_1}\cdots,x_n^{m_n} e_j$ and
\[
 [A,h_r]=\sum(m_1\lambda_1+\cdots+m_n\lambda_n-\lambda_j)\beta_{m_1,\ldots,m_{n},j}\,x_1^{m_1}\cdots x_n^{m_n} e_j \, ,
\]
one sees: Whenever $m_1\lambda_1+\cdots +m_n\lambda_n-\lambda_j\neq 0$, then one may eliminate the term $\alpha_{m_1,\ldots,m_{n},j}\,x_1^{m_1}\cdots x_n^{m_n} e_j$ in $f_r$. There remains
\[
 f_r^* = \sum_{\substack{(m_1,\ldots,m_{n},j): \\ m_1\lambda_1+\cdots+m_n\lambda_n-\lambda_j=0}} \alpha_{m_1,\ldots,m_{n},j} x_1^{m_1}\cdots x_n^{m_n} e_j\, ,
\]
which satisfies $[A_s,f_r^*]=0$. The more abstract approach taken above is helpful when $A$ is not semisimple (one avoids writing many indices), and indispensable whenever the matrix $A_s$ is not given in diagonal form.}
\end{example}

To summarize: In a step-by-step approach, for every degree $r\geq 2$ one may choose $f_r^* \in \ker({\rm ad}\,A_s)$, $\Phi=\Phi_r:=I+h_r+\cdots$, and iterating this procedure (with the sequence of compositions $\Phi_2$, $ \Phi_3\circ \Phi_2,\ldots$ being convergent in formal power series, since the orders of the changed terms tend to infinity) one obtains:
\begin{theorem}\label{nofothm}
Let  $\dot x=f(x)$ be given,  with (formal) Taylor expansion \eqref{taylorode}. Then there exist formal power series $\Phi(x)= x+\cdots$ and $f^*(x)=Ax+\sum\limits_{j\geq 2}f_j^*(x)$ such that:
\begin{itemize}
\item $\Phi$ is formally solution-preserving from $\dot x=f^*(x)$ to $\dot x=f(x)$ (i.e., the identity from Lemma \ref{solprescrit} holds in formal power series);
\item  $f^*$ is  in  {\em Poincar\'{e}-Dulac normal form}, thus $[A_s,f^*]=0$ (equivalently, all $[A_s,f_j^*]=0$).
\end{itemize}
\end{theorem} 
\begin{remark}{\em 
Theorem \ref{nofothm} is stated for formal power series. Regrettably, if one starts with convergent series, convergence may be lost in the transformation. Therefore many of the following results will be stated and proved for formal power series and formal vector fields. Since the basic notions of Lie derivative, Lie bracket etc., as well as their properties, carry over to the formal setting, this will not cause any technical problems.}
\end{remark}
We note that vector fields in Poincar\'e-Dulac normal form admit the infinitesimal symmetry $A_s$, and hence a toral symmetry group. In particular, Example \ref{toralsymex} provides systems in normal form when one sets the semisimple linear part equal to $B$.
\begin{remark} {\em Convergence problems for normal form transformations appear in many cases when the normal form is not trivial (i.e., not linear) and therefore would be interesting for the dynamics. But the step-by-step procedure underlying the proof of Theorem \ref{nofothm} shows that one may transform $f$ to normal form up to any prescribed degree $r\in\mathbb N$ (thus $[A_s,f_j^*]=0$ for all $j\leq r$) by an analytic transformation. We refer to Bruno  \cite{Br} (and the works cited therein) for an extensive discussion of convergence and divergence issues, and only give a brief sketch here:
\begin{enumerate}[(a)]
\item  Poincar\'{e} proved convergence for the case that all $\lambda_i$ are contained in an open half plane of $\mathbb C$ containing $0$ in its boundary. Much harder to prove is the following result by Siegel \cite{Si}: If there exist $\varepsilon,\nu > 0$ such that
\[
 |m_1\lambda_1+\cdots +m_n\lambda_n-\lambda_j|\geq \varepsilon (m_1+\cdots +m_n+1)^{-\nu} 
\]
for all integers $m_1,\ldots,m_n\geq 0$ and all $j$, then there exists a convergent transformation to normal form, which is necessarily of the form $\dot x=Ax$ by Example \ref{toralsymex}.
There exist examples (the first ones due to Bruno) for necessarily divergent transformations when this condition is violated.
\item Moreover there exist obstructions to convergence which are rooted in the structure of the formal normal form. For instance consider
\[ 
\dot x=f(x) = \begin{pmatrix}
				x_1 \\ -x_2
                                  \end{pmatrix}+\cdots.
\]
Any corresponding normal form has a representation
\[
f^*(x)=Ax+
\sum\varphi(x)^i(\alpha_i^* x+\beta_i^* Ax);
\]
compare Example \ref{toralsymex}. When all $\alpha_i^*=0$ then one has convergence; otherwise, there exist equations for which all normalizing transformations diverge.
\item Convergence questions and the question for existence of analytic infinitesimal symmetries are related:  If $\dot x=f(x)$ admits a convergent transformation to normal form then there exists a nontrivial infinitesimal symmetry $g\notin \mathbb K f$: If  $f^*\neq A_s$ this follows from the defining characteristic of normal forms, otherwise choose $g^*$ linear with $[g^*,A_s]=0$.  
In dimension $2$ the converse holds; see \cite{BW}: The existence of a nontrivial infinitesimal symmetry forces the existence of a convergent normalizing transformation. For partial generalizations to higher dimensions see \cite{Ci Wa}. A quite comprehensive characterization was found by Zung \cite{Zun} (see also Zung's subsequent work, especially on Hamiltonian systems).

\end{enumerate}
}
\end{remark}
Obviously, the built-in symmetries of (analytic) systems in normal form can be employed for symmetry reductions according to the previous sections. But, as we will show in the following, some remarkable properties of systems in normal form go beyond admitting symmetries.
The first result of this kind (which holds for the analytic as well as the formal case) is as follows:
\begin{theorem}\label{nofocent}
Let $f=A+\sum\limits_{j\geq 2}f_j$  be in Poincar\'e-Dulac normal form, and $g=\sum g_k$ such that $[g,f]=0$. Then $[g,A_s]=0$.
\end{theorem}
\begin{proof}
 Let  $g=g_r+\cdots$, with $g_r\neq 0$. Then $[g,f]=0$ is equivalent to
\begin{equation}\label{centcondeq}
 [A,g_{r+j}]+[f_2,g_{r+j-1}]+\cdots +[f_{j+1},g_r]=0 \qquad \text{for all}\;\,j\geq 0.
\end{equation}
We show by induction that $[A_s,g_{r+j}]=0$ for all $j\geq 0$. For $j=0$ the assertion follows from $[A,g_r]=0$ and $\ker({\rm ad}\,A)\subseteq \ker({\rm ad}\,A_s)$ on $\mathcal{P}_r$. For the induction step apply ${\rm ad}\, A_s$ to \eqref{centcondeq} and recall that $A_s$ and $f_j$ commute for all $j$. Therefore
\[
 \bigl[A,[A_s,g_{r+j}]\bigr]+\bigl[f_2,[A_s,g_{r+j-1}]\bigr]+\cdots +\bigl[f_{j+1},[A_s,g_r]\bigr]=0.
\]
By induction hypothesis all terms after the first one vanish, hence one sees that $g_{r+j}\in \ker({\rm ad}\,A_s)^2=\ker{\rm ad}\,A_s$.
\end{proof}
\begin{example}{\em 
Assume that $\dot x=f(x)$ on $U\subseteq \mathbb K^n$ admits a stationary point $y_0$, with the eigenvalues of  $Df(y_0)$ (for instance) linearly independent over $\mathbb Q$. If  $g_1,\ldots,g_r$ are vector fields with $[g_i,f]=0$ $(1\leq i\leq r)$, then the $g_i$ span an abelian Lie algebra, and $r\leq n$. To verify this ($y_0=0$ with no loss of generality), note that with  $A:=Df(0)$ any transformation to normal form yields $f^*=A$, and the $g_i^*$ commute with $A$, hence are linear and commute pairwise (see Example \ref{toralsymex}).    
Thus from the local theory one obtains restrictions on the centralizer $\mathcal{C}_U(f)$.}
\end{example}
Considering normalizers, the roles of the vector field and the infinitesimal orbital symmetry are no longer interchangeable. Therefore one has to discuss two scenarios.    \\
In the first scenario we assume that an infinitesimal orbital symmetry is in normal form. The essential result is then that the semisimple part of its linearization is itself an infinitesimal orbital symmetry.
\begin{theorem}\label{nofonorm}
Let $f=A+\cdots$ be in Poincar\'e-Dulac normal form, and furthermore let $g$ be a formal vector field, $\lambda=\lambda_0+\lambda_1+\cdots$ a formal power series such that $[f,g]= \lambda g$. Then there exists an invertible formal power series $\sigma=1+\cdots$ such that $g^*:=\sigma g$ satisfies the identity $[f,g^*]=\lambda^*g^*$, with a series $\lambda^*:=\lambda_0+\lambda_1^*+\cdots$ that in turn satisfies $X_{A_s}(\lambda^*)=0$.
Moreover  $[A_s,g^*]=\lambda_0 g^*$.  
\end{theorem}
\begin{proof} The proof is a (technically more involved) variant of the proof of Theorem \ref{nofocent}.
 Thus let  $g=g_r+\cdots$, $g_r\not=0$. By hypothesis one has in particular that
\[
 [A, g_r]=\lambda_0g_r
\]
and
\begin{equation}\label{normcondeq}
[A,g_{r+j}]+[f_2,g_{r+j-1}]+\cdots+[f_j,g_r]=\lambda_0g_{r+j}+\cdots+\lambda_j g_r.
\end{equation}
The first identity immediately shows $[A_s,g_r]=\lambda_0 g_r$.   We proceed by induction on $j$. Assuming that $X_{A_s}(\lambda_k)=0$ for all $k<j$, make the ansatz  $g^*=(1+\sigma_j)g$, with a homogeneous polynomial $\sigma_j$ of degree $j$, thus
\[
 g^*=g_r^*+g_{r+1}^*+ \cdots \,, \text{with}\;\;g_r^*=g_r,\ldots, g_{r+j-1}^* = g_{r+j-1}, g_{r+j}^*=g_{r+j}+\sigma_jg_r.
\] 
We then obtain from \eqref{normcondeq} that
\[
[A,g_{r+j}^*]+[f_2,g_{r+j-1}^*]+\cdots+[f_j,g_r^*]=\lambda_0g_{r+j}^*+\cdots+\lambda_{j-1} g_{r+1}^*+(\lambda_j+X_A(\sigma_j)g_r^*,
\]
thus we have $\lambda_k^*=\lambda_k$ for all $k<j$, and $\lambda_j^*=\lambda_j+X_A(\sigma_j)$.
The restriction of $X_A$ to the image of $X_{A_s}|_{S_j}$ is invertible, and $S_j$ is the direct sum of the kernel and image of $X_{A_s}$; hence $\sigma_j$ may be chosen such that $X_{A_s}(\lambda_j^*)=0$.\\
Since the infinite product $(1+\sigma_1)(1+\sigma_2)\cdots$ converges in formal power series, we have the first assertion.\\
Concerning the second assertion, we show by induction that all $[A_s,g_{r-j}]=\lambda_0g_{r+j}$, which is obvious for $j=0$. For the induction step, start from
\[
[A,g_{r+j}^*]+[f_2,g_r^*]+\cdots+[f_r,g_j^*]=\lambda_0g_{r+j}^* +\cdots +\lambda_j^* g_r^*
\]
and rearrange this as
\[
[A,g_{r+j}^*]-\lambda_0g_{r+j}^* =-[f_2,g_{r+j-1}^*]-\cdots-[f_j,g_r^*]+\lambda_1^* g_{r+j-1}^*+\cdots+\lambda_j^*g_r^*.
\]
With all $[A_s,f_k]=0$ and $X_{A_s}(\lambda_\ell^*)=0$, Proposition \eqref{liebproperties} and the induction hypothesis show that the right hand side lies in the  $\lambda_0$-eigenspace of ${\rm ad}\,A_s$. By semisimplicity of ${\rm ad}\,A_s$ the same holds for $g_{r+1}^*$.
\end{proof}
\begin{remark}In the scenario of Theorem \ref{nofonorm}, if  $g(0)=0$ and $\lambda_0\neq 0$, then $C:=D g(0)$ is necessarily nilpotent, by familiar facts from linear algebra (or by considering the relations in Example \ref{toralsymex}).
\end{remark}
We now reverse roles, but restrict the discussion to a special case, for which we show that the normalizer (relative to the centralizer) has a quite simple structure.
\begin{proposition}\label{nofonormprop}
Let $f= A+\cdots$ as in \eqref{taylorode}, such that the eigenvalues $\lambda_1,\ldots,\lambda_n$ of $A$ satisfy no resonance equation $\sum m_i \lambda_i-\lambda_j=0$ with all $m_i\geq 0$, $\sum m_i\geq 2$. If $g$ is a formal vector field  such that  $[g,f]=\alpha f$ with some formal series $\alpha$, then $g=h+\beta f$, with some formal series $\beta$ and $[h,f]=0$. 
\end{proposition}
\begin{proof} It is sufficient to prove the assertion when $f=f^*$ is in normal form, and this normal form is just $f^*=A$. With $g=C+g_2+\cdots$,  $[g,A]=\alpha A$ is equivalent to $[C,A]=0$ and $[g_m,A]=\alpha_{m-1}A$ for all $m\geq 2$. (Since the semisimple part of $A$ is nontrivial, $[C,A]=\alpha_0 A$ is only possible for $\alpha_0=0$.) For degree $m>1$ the assumption on the eigenvalues implies that ${\rm ad }\,A|_{\mathcal{P}_m}$ and $X_A|_{S_{m-1}}$ are invertible.
Thus for every $\alpha_{m-1}\in S_{m-1}$ there exists $\beta_{m-1}$ with $X_A(\beta_{m-1})=-\alpha_{m-1}$; we obtain $[\beta_{m-1} A,A]=\alpha_{m-1} A$, hence $[g_m-\beta_{m-1} A,A]=0$ and $g_m=\beta_{m-1} A$ by invertibility of ${\rm ad}\,A$.
\end{proof}
For a more general characterization see \cite{KWZ}, Thm.~2.6.

\medskip
After discussing the special properties of normal forms with respect to Lie brackets, there remains to discuss special properties of normal forms with respect to Lie derivatives. The following facts concerning first integrals and semi-invariants of normal forms (of formal vector fields) are proven in a very similar manner to Theorems \ref{nofocent} and \ref{nofonorm}. (A proof for part (b) was first given in \cite{WPoi}, Lemma 2.2.)

\begin{theorem}\label{nofosemi}Let $f=A+\cdots$ be in Poincar\'e-Dulac normal form, and $\phi=\sum_{k>0} \phi_k$ a formal power series with zero constant term.
\begin{enumerate}[(a)]
\item If $\phi$ is a first integral of $f$ then $\phi$ is also a first integral of $\dot x=A_sx$.
\item If $\phi$ is a semi-invariant of $f$, thus $X_f(\phi)=\beta\phi$ for some series $\beta$, then there exists an invertible series $\sigma=1+\cdots$ such that $\phi^*:=\sigma\cdot\phi$ is a semi-invariant of $f$, with $X_f(\phi^*)=\beta^*\phi^*$, and $\beta^*=\beta^*_0+\beta^*_1+\cdots$ satisfying $X_{A_s}(\beta^*)=0$.\\
Moreover $X_{A_s}(\phi^*)=\beta_0^*\phi^*$; in particular $\phi^*$ (as well as $\phi$) is a semi-invariant of $\dot x=A_sx$.
\end{enumerate}

\end{theorem}
Note that $\phi$ is a semi-invariant of $f$ if and only if the ideal $\left<\phi\right>$ of the formal power series algebra is invariant with respect to $X_f$. Moreover, if $\phi$ and $f$ are convergent then the zero set of $\phi$ is invariant for the diffeerntial equation. There exists a natural generalization of Theorem \ref{nofosemi}(b) to invariant ideals with an arbitrary set of generators. We only consider the complex setting here, for the sake of simplicity; for more background see also Chapter 5.

We denote by $ \mathbb C[[x_1,\ldots,x_n]]$ the commutative and associative algebra of formal power series, and by  $ \mathbb C[[x_1,\ldots,x_n]]_c$ the subalgebra of power series with a nonempty domain of convergence. A local analytic set near $0$ may be identified with the common zero set of certain elements of $ \mathbb C[[x_1,\ldots,x_n]]_c$ that vanish at $0$ (finitely many suffice due to the Noetherian property of this ring). By the Hilbert-R\"uckert Nullstellensatz there is a 1-1 correspondence between analytic sets and radical ideals of  $ \mathbb C[[x_1,\ldots,x_n]]_c$. Now, if $f$ is analytic in $0$ and $f(0)=0$ then an analytic set with vanishing ideal $J$ is invariant for $\dot x=f(x)$ if and only if $X_f(J)\subseteq J$, as was shown in Proposition \ref{invarcritprop}. While one cannot sensibly extend the notion of invariant set to formal power series and vector fields, the notion of invariant ideals carries over, and this yields partial information on analytic invariant sets. This is the setting we will consider now. 

The following are two main results from N.~Kruff's doctoral dissertation, published in \cite{NKr}. The proofs are too long (and a bit too technically involved) to be presented here.

\begin{theorem} \label{kruffthm1}Let  $f=A+\cdots$ be a complex formal power series vector field in $n$ variables, in Poincar\'e-Dulac normal form. Then every ideal $J\subseteq  \mathbb C[[x_1,\ldots,x_n]]$ that is invariant with respect to $X_f$ is also invariant with respect to $X_{A_s}$.
\end{theorem}

This result (which corresponds to the first statement of Theorem \ref{nofosemi}(b)) is complemented by the following characterization of invariant ideals  (which may be said to correspond to the remaining statements of Theorem \ref{nofosemi}(b)).
\begin{theorem}\label{kruffthm2}
Let $B={\rm diag}\,\left(\mu_1,\ldots,\mu_n\right)$. Then the following hold.
\begin{enumerate}[(a)]
\item Every $X_B$-invariant ideal of $\mathbb C[[x_1,\ldots,x_n]]$ admits a generator system consisting of semi-invariants of $B$.
\item Up to multiplication by an invertible series, for every semi-invariant $\phi$  of $B$ there exists $\alpha\in\mathbb C$ such that $\phi$ is a series in monomials $x_1^{d_1}\cdots x_n^{d_n}$ with $\sum d_i\mu_i=\alpha$.
\end{enumerate}
\end{theorem}
These properties facilitate the investigation of invariant ideals (and thus possible invariant analytic sets).
\begin{example} {\em Let $B={\rm diag}\,\left(\mu_1,\ldots,\mu_n\right)$.
\begin{itemize}
\item Whenever $\mu_1,\ldots,\mu_n$ are linearly independent over the rationals $\mathbb Q$ then $\phi_1=x_1,\ldots,\phi_n=x_n$ are the only irreducible semi-invariants for $B$ (up to multiplication with invertible series). To see this, note that  $(d_1,\ldots,d_n)$ are uniquely determined by $\alpha=\sum d_i\mu_i$, hence the linear combination of monomials in part (b) of the Theorem contains only one term and is reducible when $\sum d_i>1$. By Theorem \ref{kruffthm2}, the only $X_B$-invariant prime ideals in  $\mathbb C[[x_1,\ldots,x_n]]$ have the form $\left<x_{i_1},\ldots,x_{i_r}\right>$.
\item If $\mathbb Q\mu_1+\cdots +\mathbb Q\mu_n$ has dimension $n-1$ over $\mathbb Q$ and there exist positive, relatively prime integers $m_1,\ldots,m_n$ with $\sum m_i\mu_i=0$ then, again, $\phi_1=x_1,\ldots,\phi_n=x_n$ are the only irreducible semi-invariants for $B$ (up to multiplication with invertible series). And again the only $X_B$-invariant prime ideals in  $\mathbb C[[x_1,\ldots,x_n]]$ have the form $\left<x_{i_1},\ldots,x_{i_r}\right>$. To see this, first note that two monomials $x_1^{d_1}\cdots x_n^{d_n}$ and $x_1^{e_1}\cdots x_n^{e_n}$ contribute to a semi-invariant with the same cofactor $\alpha\in\mathbb C$ if and only if 
\[
(d_1-e_1,\ldots,d_n-e_n)\in \mathbb Z\,(m_1,\ldots, m_n).
\]
From this one obtains with induction that every semi-invariant is (up to multiplication by an invertible series) of the form
\[
\text{ monomial}\times \sum_i\gamma_i\left(x_1^{m_1}\cdots x_n^{m_n}\right)^i,
\]
which implies the assertion. By Theorem \ref{kruffthm1} this characterization carries over to $X_f$-invariant ideals, for any $f=B+\cdots$ in normal form. 
\item The two cases above may be seen as exceptions: In general a linear vector field will admit infinitely many (not associated) irreducible semi-invariants.
\end{itemize}
}
\end{example}
\section{Further remarks and notes}
\begin{itemize}
\item The literature on compact symmetry groups and their various applications is vast; we mention only Abud and Sartori \cite{AS}, Golubitsky and coauthors \cite{GS,GSS}, Marsden and Ratiu \cite{MR}. An important paper on dynamical systems with a compact symmetry group is due to Field \cite{Fi}. A standard reference to linear algebraic groups is Borel \cite{Bo}. Regarding normal forms, in particular convergence and divergence questions, Bruno's work (see \cite{Br} and the references therein) is fundamental. The coordinate-free approach using Lie brackets was detailed in \cite{Wa 3}.
\item Turning to the question what can be salvaged for the smooth case, the first answer is ``everything in 4.1'' provided one considers only compact (linear) symmetry groups. An essential ingredient is a theorem of Schwarz \cite{Sc 1} that corresponds to (and actually predates) Theorem \ref{anaredbyinv}, together with an extension by Poenaru \cite{Po} to vector fields. As for normal forms, power series represent smooth functions up to a flat remainder. By a result of Chen \cite{Che}, there is always a smooth normalizing transformation provided that the linearization of the vector field is hyperbolic, i.e., does not possess any eigenvalues with zero real parts.
\item For more recent results about centralizers, normalizers and Jacobi multipliers of local analytic and formal vector fields see \cite{KWZ}.
\end{itemize}
\chapter{Appendix: Some background on power series}
We collect here some facts on power series and analytic functions. Concerning algebraic properties, more details can be found in Zariski and Samuel \cite{ZaSa60}, Ruiz \cite{Rui}  and Shafarevich \cite{Sha}; for convergence properties see Cartan \cite{Ca}. The presentation uses a rather broad brush; see the literature for details.
\subsection*{Formal and convergent power series}
\begin{enumerate}
\item A (formal) power series in $n$ variables about $b=(b_1,\ldots, b_n)^{\rm tr}\in \mathbb K^n$ has the form
\[
\psi(x)=\sum _{i_1,\ldots,i_n\geq 0}a_{i_1,\ldots,i_n}\,(x_1-b_1)^{i_1}\cdots (x_n-b_n)^{i_n}=\sum_{I\in\mathbb N_0^n}a_I(x-b)^I.
\]
By definition this is just the (multi-index) sequence of the partial sums. Usually we will take $b=0$. One may then also write this series in the form 
\[
\psi(x)=\sum_{k\geq 0}\psi_{k}(x)
\]
with each $\psi_k$ a homogeneous polynomial in
$x_1,\ldots,\,x_n$ of degree $k$. The formal power series (with the Cauchy product as multiplication) form a commutative $\mathbb K$-algebra which we denote by $\mathbb K[[x_1,\ldots,\,x_n]]$.
\item For series with multiple indices one has a good notion of absolute convergence: Take any bijection $\tau: \,\mathbb N_0\to\mathbb N_0^n$ and consider the series
\[
\sum_{k\geq 0}a_{\tau(k)}x^{\tau(k)}
\]
The notions of absolute convergence and limit may then be defined by the corresponding properties of the latter series. This is unambiguous since choosing a different bijection would amount to a rearrangement of terms, which does not affect convergence properties or limits.
\item Convergence domains of power series: If there is a $z=(z_1,\ldots,z_n)^{\rm tr}\in {\mathbb K^*}^n$ such that the series $\sum a_Iz^I$ converges absolutely then the power series converges absolutely and locally uniformly in the open set defined by $|x_1|<|z_1|,\ldots, |x_n|<|z_n|$, and thus defines a function on this set. We denote the power series which converge on some nonempty open set by $\mathbb K[[x_1,\ldots,\,x_n]]_c$; these form a subalgebra of  $\mathbb K[[x_1,\ldots,\,x_n]]$. (We will briefly speak of {\em convergent power series}.)
\item Rearrangement of summation: Given a power series with nonempty domain $\Delta: |x_i|<\rho_i, \,1\leq i\leq n$ of convergence, for every $r$ with $1\leq r<n$ one has the identity
\[
\sum a_I\,x^I=\sum_{i_{r+1},\ldots,i_n\geq 0}\left(\sum _{i_1,\ldots,i_r\geq 0}a_{i_1,\ldots,i_n}\,x_1^{i_1}\cdots x_r^{i_r}\right)x_{r+1}^{i_{r+1}}\cdots x_n^{i_n};
\]
thus each term in brackets converges absolutely and locally uniformly for $|x_1|<\rho_1,\ldots,\,|x_r|<\rho_r$, and for fixed $x_1,\ldots,x_r$ in this domain the remaining series in $x_{r+1},\ldots,x_n$ converges absolutely and locally uniformly in $|x_{r+1}|<\rho_{r+1},\ldots,|x_n|<\rho_n$.
\item Differentiation: Given a power series $\sum a_Ix^I$ with nonempty open domain of convergence (defining a function $\psi$), the series of partial derivatives (of any order) converges on the same domain and represents the corresponding partial derivative of $\psi$. In particular 
\[
a_{k_1,\ldots,k_n}=\frac{\partial^{k_1+\cdots k_n}\psi}{\partial x_1^{k_1}\cdots \partial x_n^{k_n}}(0).
\]
\item Analytic functions: Given an open subset $U$ of $\mathbb K^n$ and a function $\phi:\,U\to\mathbb K$, one calls $\phi$ analytic if at every $b\in U$ there exists a power series about $b$ which has nonempty domain of convergence and represents $\phi$ on this domain. (The definition extends to maps from $U$ to $\mathbb K^m$.) One has the identity theorem: If $U$ is connected then $\phi=0$ if and only if $\phi|_{U^*}=0$ for some open and nonempty $U^*\subseteq U$.
\item If one is interested in the behavior of analytic functions at a specific point (w.l.o.g. $0$) then the appropriate structure is the algebra of germs $(V,\,\phi)$, with $V$ a neighborhood of $0$ and $\phi:\,V\to \mathbb K$ analytic. One identifies two germs if representing functions agree in some neighborhood of $0$. This algebra is canonically isomorphic to $\mathbb K[[x_1,\ldots,\,x_n]]_c$.
\item There are also relevant notions of convergence for formal power series, related to the ${\bf m}$-adic topology: Given nonzero $\psi=\sum \psi_k$, the order of $\psi$ is defined as the smallest index $\ell$ with $\psi_\ell\not=0$, and denoted by $o(\psi)$. (The order of $0$ may be defined as $\infty$.) With respect to the ${\bf
m}$-adic topology, a sequence
$\sum \gamma_k$ of formal power series converges to a formal power series $\gamma$ if and only if the orders of $\gamma_k-\gamma $ tend
to infinity as $k\to\infty$. Note: Every sequence of formal power series with the property that $o(\gamma_k-\gamma_\ell)$ tends to infinity as $k$ and $\ell$ tend to infinity (a Cauchy sequence) is convergent. 
\end{enumerate}
\subsection*{Some algebraic properties}
\begin{enumerate}
\item 
 A formal power series
$\psi$ is multiplicatively invertible in $\mathbb K[[x_1,\ldots,
\,x_n]]$ if and only if  its order equals
$0$. There is a unique maximal ideal ${\bf m}:=\left\{\phi:\,
\phi_0 =0\right\}$. The same properties hold for $\mathbb K[[x_1,\ldots,
\,x_n]]_c$.
\item The ring of formal power series is Noetherian and
a unique factorization domain, thus every ideal is finitely generated and every nonzero non--invertible
series is a product of irreducible ones, the representation being
unique up to the ordering of factors and multiplication by
invertible series. Mutatis mutandis, these properties carry over to
$\mathbb K[[x_1,\ldots, \,x_n]]_c$.
\item Analytic subvarieties of $\mathbb K^n$: Locally an analytic subvariety $Y$ is given as the common zero set of finitely many analytic functions $\phi_1,\ldots,\phi_r$. If one is interested in a particular point (w.l.o.g. $0$, as usual) then one should (introduce an obvious equivalence relation and) consider germs of such varieties and germs of defining functions. Then it is natural to associate the ideal $\left<\phi_1,\ldots,\phi_r\right>\subseteq \mathbb K[[x_1,\ldots,\,x_n]]_c$ to $Y$ and to additionally consider the vanishing ideal
\[
\mathcal{J}(Y):=\left\{\psi\in \mathbb K[[x_1,\ldots, \,x_n]]_c; \,\psi|_Y=0\right\}.
\]
Note that this is a radical ideal, thus $\psi^m\in\mathcal{J}(Y)$ for some $m>0$ implies that  $\psi\in\mathcal{J}(Y)$.
\item Hilbert--R\"uckert Nullstellensatz: For $\mathbb K=\mathbb C$ there is a $1-1$ correspondence between local analytic subvarieties and radical ideals of $ \mathbb C[[x_1,\ldots, \,x_n]]_c$: Assigning to each radical ideal $J$ its vanishing set $\mathcal{V}(J)$, the equality $\mathcal{J}(\mathcal{V}(J))=J$ holds.

\end{enumerate}


\end{document}